\documentclass[final,11pt,a4paper,reqno]{amsart} 

\usepackage{soul}
\usepackage[dvipsnames]{xcolor}
\usepackage{graphicx}
\usepackage[all]{xy}
\usepackage{placeins}
\usepackage{enumitem}
\usepackage{amssymb}
\usepackage{latexsym}
\usepackage{amsmath}
\usepackage{mathrsfs}
\usepackage{array,booktabs}
\usepackage{verbatim}
\usepackage{fullpage}
\usepackage{textcomp}
\usepackage{color}
\usepackage{cancel}
\usepackage{amsfonts}
\usepackage{showkeys}
\usepackage{psfrag,epsfig}
\usepackage{mathtools}
\usepackage{float}

\usepackage{amsfonts,amsmath,amssymb,amsthm}
\usepackage{helvet}
\usepackage{pdfpages}
\usepackage{tikz}
\usetikzlibrary{arrows,decorations,shapes,positioning,patterns}
\usepackage{multicol}
\usepackage{hyperref}
\usepackage{soul}
\usepackage{MnSymbol}
 
\newcommand{\T}{\mathcal{T}} 
\newcommand{\Sur}{\cals} 
\newcommand{\M}{M} 
\newcommand{\x}{\mathbf{x}} 
\newcommand{\y}{\mathbf{y}} 
\newcommand{\A}{\mathcal{A}} 
\newcommand{\Aq}{\mathcal{A}_q} 
\newcommand{\PM}{P} 
\newcommand{\Match}{\operatorname{Match}} 
\newcommand{\SG}{\mathcal{G}} 
\newcommand{\Gn}{\SG_n} 
\newcommand{\SH}{\mathcal{H}} 
\newcommand{\Hn}{\SH_n} 
\newcommand{\inv}{\sigma} 
\newcommand{\QTorus}{T} 

\newcommand{\ke}{\textcolor{red}}
\newcommand{\Philipp}{\textcolor{blue}}
\newcommand{\new}{\textcolor{olive}}
\usepackage{hyperref}
\newcommand{\harxiv}[1]{ \href{http://arxiv.org/abs/#1}{\texttt{arXiv:#1}}}
\newcommand{\hyref}[2]{ \hyperref[#2]{#1~\ref*{#2}} }

\newcommand{\coloneqq}{\mathrel{\mathop:}=}
\newcommand{\eqqcolon}{=\mathrel{\mathop:}}

\theoremstyle{definition}
\newtheorem{theorem}{Theorem}[section]
\newtheorem{lemma}[theorem]{Lemma}
\newtheorem{prop}[theorem]{Proposition}
\newtheorem{defn}[theorem]{Definition}

\newtheorem{cor}[theorem]{Corollary}

\newtheorem{question}[theorem]{Question}

\newtheorem{rem}[theorem]{Remark}
\newtheorem{ex}[theorem]{Example}
\newtheorem{notation}[theorem]{Notation}

\newtheorem{thmIntro}{Theorem}    \renewcommand{\thethmIntro}{\Alph{thmIntro}}
\newtheorem{questIntro}{Question}    \renewcommand{\thequestIntro}{\Alph{questIntro}}
    \renewcommand{\theconjIntro}{\Alph{conjIntro}}
\usepackage{stackengine,graphicx,amsmath,fp,scalerel}
\newcount\crosswd
\newcount\termwd
\newcommand\rawcrossout[2]{\ensurestackMath{%
  \setbox0=\hbox{$#2$}%
  \crosswd=\wd0\relax%
  \setbox0=\hbox{$#1$}%
  \termwd=\wd0\relax%
  \FPdiv\myscale{\the\termwd}{\the\crosswd}%
  \stackengine{0pt}{#1}{\stretchrel*{\scalebox{\myscale}[1]{#2}}{#1}}{O}{c}{F}{T}{L}}}
\def\XX{\kern-3pt/}
\def\YY{\kern-.5pt}
\newcommand\crossout[1]{\rawcrossout{#1}{\YY/\YY}}
\newcommand\dcrossout[1]{\rawcrossout{#1}{\YY/\XX\YY}}
\newcommand\tcrossout[1]{\rawcrossout{#1}{\YY/\XX\XX\YY}}
\newcommand\qcrossout[1]{\rawcrossout{#1}{\YY/\XX\XX\XX\YY}}

\newcommand{\Canakci}{\c{C}anak\c{c}\i}
\newcommand{\Ilke}{\.{I}lke }
\newcommand{\I}{\.{I}}

\renewcommand*{\theintrotheorem}{\Alph{introtheorem}}

\definecolor{dgreen}{rgb}{0,0.45,0}
\definecolor{beaublue}{rgb}{0.74, 0.83, 0.9}
\definecolor{darklavender}{rgb}{0.45, 0.31, 0.59}
\definecolor{darkorchid}{rgb}{0.6, 0.2, 0.8}
\definecolor{darkpastelpurple}{rgb}{0.59, 0.44, 0.84}
\definecolor{electricviolet}{rgb}{0.56, 0.0, 1.0}

\font\sss=cmss8

\newcommand{\sA}{\mathsf{A}}
\newcommand{\sB}{\mathsf{B}}
\newcommand{\sC}{\mathsf{C}}
\newcommand{\sD}{\mathsf{D}}
\newcommand{\sE}{\mathsf{E}}
\newcommand{\sF}{\mathsf{F}}
\newcommand{\sG}{\mathsf{G}}
\newcommand{\sH}{\mathsf{H}}
\newcommand{\sI}{\mathsf{I}}
\newcommand{\sJ}{\mathsf{J}}
\newcommand{\sK}{\mathsf{K}}
\newcommand{\sL}{\mathsf{L}}
\newcommand{\sM}{\mathsf{M}}
\newcommand{\sN}{\mathsf{N}}
\newcommand{\sO}{\mathsf{O}}
\newcommand{\sP}{\mathsf{P}}
\newcommand{\sQ}{\mathsf{Q}}
\newcommand{\sR}{\mathsf{R}}
\newcommand{\sS}{\mathsf{S}}
\newcommand{\sT}{\mathsf{T}}
\newcommand{\sU}{\mathsf{U}}
\newcommand{\sV}{\mathsf{V}}
\newcommand{\sW}{\mathsf{W}}
\newcommand{\sX}{\mathsf{X}}
\newcommand{\sY}{\mathsf{Y}}
\newcommand{\sZ}{\mathsf{Z}}

\DeclareMathAlphabet{\mathpzc}{OT1}{pzc}{m}{it}

\newcommand{\bA}{\mathbb{A}}
\newcommand{\bB}{\mathbb{B}}
\newcommand{\bC}{\mathbb{C}}
\newcommand{\bD}{\mathbb{D}}
\newcommand{\bE}{\mathbb{E}}
\newcommand{\bF}{\mathbb{F}}
\newcommand{\bG}{\mathbb{G}}
\newcommand{\bH}{\mathbb{H}}
\newcommand{\bI}{\mathbb{I}}
\newcommand{\bJ}{\mathbb{J}}
\newcommand{\bK}{\mathbb{K}}
\newcommand{\bL}{\mathbb{L}}
\newcommand{\bM}{\mathbb{M}}
\newcommand{\bN}{\mathbb{N}}
\newcommand{\bO}{\mathbb{O}}
\newcommand{\bP}{\mathbb{P}}
\newcommand{\bQ}{\mathbb{Q}}
\newcommand{\bR}{\mathbb{R}}
\newcommand{\bS}{\mathbb{S}}
\newcommand{\bT}{\mathbb{T}}
\newcommand{\bU}{\mathbb{U}}
\newcommand{\bV}{\mathbb{V}}
\newcommand{\bW}{\mathbb{W}}
\newcommand{\bX}{\mathbb{X}}
\newcommand{\bY}{\mathbb{Y}}
\newcommand{\bZ}{\mathbb{Z}}

\newcommand{\ga}{\calg_{\gamma}}
\DeclareMathOperator{\match}{\textup{ Match }}
\DeclareMathOperator{\cross}{\textup{cross }}
\DeclareMathOperator{\acc}{\mathrm{Acc }}
\newcommand{\cala}{\mathcal{A}}
\newcommand{\calb}{\mathcal{B}}
\newcommand{\calc}{\mathcal{C}}
\newcommand{\cald}{\mathcal{D}}
\newcommand{\cale}{\mathcal{E}}
\newcommand{\calf}{\mathcal{F}}
\newcommand{\calg}{\mathcal{G}}
\newcommand{\calh}{\mathcal{H}}
\newcommand{\cali}{\mathcal{I}}
\newcommand{\calj}{\mathcal{J}}
\newcommand{\calk}{\mathcal{K}}
\newcommand{\call}{\mathcal{L}}
\newcommand{\calm}{\mathcal{M}}
\newcommand{\caln}{\mathcal{N}}
\newcommand{\calo}{\mathcal{O}}
\newcommand{\calp}{\mathcal{P}}
\newcommand{\calq}{\mathcal{Q}}
\newcommand{\calr}{\mathcal{R}}
\newcommand{\cals}{\mathcal{S}}
\newcommand{\calt}{\mathcal{T}}
\newcommand{\calu}{\mathcal{U}}
\newcommand{\calv}{\mathcal{V}}
\newcommand{\calw}{\mathcal{W}}
\newcommand{\calx}{\mathcal{X}}
\newcommand{\caly}{\mathcal{Y}}
\newcommand{\calz}{\mathcal{Z}}
\newcommand{\tA}{\mathtt{A}}
\newcommand{\tB}{\mathtt{B}}
\newcommand{\tC}{\mathtt{C}}
\newcommand{\tD}{\mathtt{D}}
\newcommand{\tE}{\mathtt{E}}
\newcommand{\tF}{\mathtt{F}}
\newcommand{\tG}{\mathtt{G}}
\newcommand{\tH}{\mathtt{H}}
\newcommand{\tI}{\mathtt{I}}
\newcommand{\tJ}{\mathtt{J}}
\newcommand{\tK}{\mathtt{K}}
\newcommand{\tL}{\mathtt{L}}
\newcommand{\tM}{\mathtt{M}}
\newcommand{\tN}{\mathtt{N}}
\newcommand{\tO}{\mathtt{O}}
\newcommand{\tP}{\mathtt{P}}
\newcommand{\tQ}{\mathtt{Q}}
\newcommand{\tR}{\mathtt{R}}
\newcommand{\tS}{\mathtt{S}}
\newcommand{\tT}{\mathtt{T}}
\newcommand{\tU}{\mathtt{U}}
\newcommand{\tV}{\mathtt{V}}
\newcommand{\tW}{\mathtt{W}}
\newcommand{\tX}{\mathtt{X}}
\newcommand{\tY}{\mathtt{Y}}
\newcommand{\tZ}{\mathtt{Z}}
\newcommand{\ba}{\bar{a}}
\newcommand{\bb}{\bar{b}}
\newcommand{\bc}{\bar{c}}
\newcommand{\bd}{\bar{d}}
\newcommand{\be}{\bar{e}}
\newcommand{\bg}{\bar{g}}
\newcommand{\bh}{\bar{h}}
\newcommand{\bi}{\bar{i}}
\newcommand{\bj}{\bar{j}}
\newcommand{\bk}{\bar{k}}
\newcommand{\bl}{\bar{l}}
\newcommand{\bm}{\bar{m}}
\newcommand{\bn}{\bar{n}}
\newcommand{\bo}{\bar{o}}
\newcommand{\bp}{\bar{p}}
\newcommand{\bq}{\bar{q}}
\newcommand{\br}{\bar{r}}
\newcommand{\bs}{\bar{s}}
\newcommand{\bt}{\bar{t}}
\newcommand{\bu}{\bar{u}}
\newcommand{\bv}{\bar{v}}
\newcommand{\bw}{\bar{w}}
\newcommand{\bx}{\bar{x}}
\newcommand{\by}{\bar{y}}
\newcommand{\bz}{\bar{z}}
\newcommand{\Db}{\sD^b}
\newcommand{\KminusL}{\sK^{b,-}(\proj \Lambda)}

\newcommand{\kk}{{\mathbf{k}}}

\renewcommand{\setminus}{\backslash}

\DeclareMathOperator{\Ex}{\mathrm{Ex}}

\DeclareMathOperator{\coker}{\mathrm{coker}}
\DeclareMathOperator{\kernel}{\mathrm{ker}}
\DeclareMathOperator{\im}{\mathrm{im}}
\DeclareMathOperator{\Hom}{\mathrm{Hom}}
\DeclareMathOperator{\Ext}{\mathrm{Ext}}

\DeclareMathOperator{\Tw}{\mathrm{Twist}}

\renewcommand{\mod}[1]{\mathsf{mod}(#1)}
\newcommand{\proj}[1]{\mathsf{proj}(#1)}

\newcommand{\too}{\longrightarrow}
\newcommand{\rightlabel}[1]{\stackrel{#1}{\longrightarrow}}
\DeclareMathOperator{\bad}{\mathrm{Bad }}
\newcommand{\xydot}{{\bullet}}
\newcommand{\arr}{\ar@{-}[r]}

\newcommand{\smxy}[1]{{\text{\tiny$#1$}}}

\newcommand{\mun}{\mu^{(n)}}
\newcommand{\muone}{\mu^{(1)}}
\newcommand{\muze}{\mu^{(0)}}
\newcommand{\muinf}{\mu^{\infty}}
\newcommand{\mubul}{\mu^{\bullet}}
\renewcommand{\phi}{\varphi}
\renewcommand{\epsilon}{\varepsilon}
\usepackage{tikz}
\usepackage{tikz-cd}
\usetikzlibrary{snakes}
\usetikzlibrary{shapes.geometric,positioning}
\usetikzlibrary{arrows,decorations.pathmorphing,decorations.pathreplacing}
\usetikzlibrary{matrix, calc, arrows}
\usetikzlibrary{decorations.pathmorphing,shapes}
\usetikzlibrary{decorations.pathreplacing}
\newcounter{sarrow}
\newcommand\xrsquigarrow[1]{%
	\stepcounter{sarrow}%
	\begin{tikzpicture}[decoration=snake]
	\node (\thesarrow) {\strut#1};
	\draw[->,decorate] (\thesarrow.south west) -- (\thesarrow.south east);
	\end{tikzpicture}%
}

\tikzset{join/.code=\tikzset{after node path={%
			\ifx\tikzchainprevious\pgfutil@empty\else(\tikzchainprevious)%
			edge[every join]#1(\tikzchaincurrent)\fi}}}

\tikzset{>=stealth',every on chain/.append style={join},
	every join/.style={->}}

\tikzset{vertex/.style={circle,fill=black,inner sep=1pt,outer sep=2pt},
	tinyvertex/.style={font=\scriptsize,minimum size=6pt},
	smallvertex/.style={inner sep=1pt, font=\small},
	>=stealth',
	leadsto/.style={-angle 90,decorate,decoration=snake,very thick},
	cut/.style={decorate,decoration=saw,very thick}}
\tikzset{
	partial ellipse/.style args={#1:#2:#3}{
		insert path={+ (#1:#3) arc (#1:#2:#3)}
	}
}
\begin{document}

\title[An expansion formula for type $A$ and Kronecker quantum cluster algebras]{An expansion formula for type $A$ and Kronecker quantum cluster algebras}
\thanks{The first author was supported by EPSRC grant EP/P016014/1 and the second author was supported by EPSRC grants EP/N005457/1 and EP/M004333/1.}

\author{\Ilke  \Canakci}
\address{School of Mathematics, Statistics and Physics, Newcastle University, Newcastle Upon Tyne NE1 7RU, United Kingdom}
\email{ilke.canakci@newcastle.ac.uk}

\author{Philipp Lampe}
\address{School of Mathematics, Statistics and Actuarial Science (SMSAS), Canterbury CT2 7FS, United Kingdom
}
\email{P.B.Lampe@kent.ac.uk}

\begin{abstract} We introduce an expansion formula for elements in quantum cluster algebras associated to type $A$ and Kronecker quivers with principal quantization. Our formula is parametrized by perfect matchings of snake graphs as in the classical case. In the Kronecker type, the coefficients are $q$-powers whose exponents are given by a weight function induced by the lattice of perfect matchings. As an application, we prove that a reflectional symmetry on the set of perfect matchings satisfies Stembridge's $q=-1$ phenomenon with respect to the weight function. Furthermore, we discuss a relation of our expansion formula to generating functions of BPS states. 
\end{abstract}

\maketitle

{\small
\setcounter{tocdepth}{1}
\tableofcontents
}
\section{Introduction}

Cluster algebras were introduced by Fomin--Zelevinsky \cite{FZ,FZ2,FZ4} and by Berenstein--Fomin--Zelevinsky \cite{BFZ} in a series of four articles to give an algebraic framework for the study of total positivity and dual canonical bases in Lie theory. Cluster algebras occur naturally in many areas of mathematics such as representation theory, geometry, combinatorics and mathematical physics.

An important class of cluster algebras, namely \emph{surface cluster algebras}, was introduced by Fomin--Shapiro--Thurston~\cite{FST} using combinatorics of (oriented) surfaces with marked points. Subsequently, Fomin--Thurston~\cite{FT} gave an intrinsic formulation of surface cluster algebras, building on earlier works of Gekhtman--Shapiro--Vainshtein \cite{GSV}, Fock--Goncharov \cite{FG1,FG2} and Penner \cite{P}, considering hyperbolic structures on the surface. Surface cluster algebras are also important from a classification point of view since all but finitely many cluster algebras of finite mutation type are those associated to marked surfaces or of rank $2$ by Felikson--Shapiro--Tumarkin~\cite{FST2,FST3}.

Quantum deformations of cluster algebras were introduced  by Berenstein--Zelevinsky~\cite{BZ} to develop a setting for a general notion of canonical bases. Their construction also builds on Fock--Goncharov~\cite{FG1,FG2} but is given in a more systematic way and, in particular, they show most structural results of cluster algebras in the quantum setting. 

Combinatorial, geometric and representation theoretic aspects of surface cluster algebras have been studied significantly in the classical case. An expansion formula for elements in surface cluster algebras was given by \cite{MSW} extending on \cite{S,ST,MS} and further combinatorial aspects of surface cluster algebras were explored in \cite{CS1,CS2,CS3,CS4,CLS}. For some results on representation theoretic aspects of surface cluster algebras, see \cite{AP,BZh,CaSc,QZ,ZZZ}. In the quantum setting most of research goes into understanding cluster variables and a `good' basis in representation theoretic terms (see for instance \cite{KQ,Q1,Q2,Q3,Q4,R}), into identifying quantum cluster algebra structures in Lie theory (see \cite{GLS,GY,G,GL}) or into the so-called positivity conjecture which is shown to be true in \cite{D} for all quantum cluster algebras. Representation theoretic expressions for quantum cluster variables were given by \cite{R,R2}; combinatorial expressions for quantum cluster variables were given by \cite{L,L2}; geometric descriptions for quantum cluster variables were given by \cite{M}. Also see \cite{BR} for the non-commutative version of cluster algebras.
  
\bigskip

In the quantum setting, there is not much known solely for surface type quantum cluster algebras. The aim of this article is to take this direction and extend the expansion formula to give explicit Laurent polynomial expressions for quantum cluster variables in the quantum cluster algebra of type $A_n$ and for quantum cluster variables and certain other elements in the quantum cluster algebra of Kronecker type considered with principal quantization. More precisely, for surface type cluster algebras and for their quantum analogues, cluster variables (and quantum cluster variables) are in bijection with (isotopy classes) of arcs in the surface. Furthermore, there is a planar graph, called \emph{snake graph}, associated to every arc in the surface. The explicit formula of ~\cite{MSW} for elements in the cluster algebra is parameterized by perfect matchings of snake graphs. In this article we show that they can also be used to give Laurent polynomial expansions for elements in the quantum cluster algebra. 

To state our main result, we introduce some notation. The quantum deformation of a cluster algebra is obtained by making each cluster into a \emph{quasi-commuting} family of variables. Then the quantum cluster algebra $\cala_q$ is a $\mathbb{Z}[q^{\pm 1/2}]$-subalgebra of the skew-field of fractions of the \emph{based quantum torus} (see Section~\ref{Sec:QCA}). The based quantum torus admits a $\mathbb{Z}[q^{\pm 1/2}]$-basis indexed by $\mathbb{Z}^m$, denoted $\{\, M[a]\mid a\in \mathbb{Z}^m\,\}$.

Let $\cala_q$ be the quantum cluster algebra with principal quantization associated with an adjacency quiver $Q$ of a triangulation of a surface $\cals$ which is either a disc with (finitely) many marked points on the boundary or an annulus with exactly one marked point in each boundary component. 

\begin{thmIntro}[Theorem~\ref{Thm:ExpansionA}, Theorem~\ref{Thm:KroneckerSnakes}] \label{ThmA} With the notation above, let $\gamma$ be a (generalized) arc in $\cals$, $x_{\gamma}$ be the corresponding element in $\cala_q$ and $\calg$ be the snake graph associated with $\gamma$. Then
\[
x_{\gamma} = \sum_{\PM\models\SG} q^{\Omega(P)} M\left[\nu(\PM)\right]
\]
where $\nu(\PM)\in\mathbb{Z}^m$ is a vector associated with a perfect matching $P$ of $\SG$ and where $\Omega(P)=0$ if $\cals$ is a disc and $\Omega(P)$ is a half-integer given explicitly by the position of the perfect matching in the perfect matching lattice if $\cals$ is the annulus with exactly one marked point in each of the boundary components.
\end{thmIntro}

We note that our formula coincides with the classical case \cite{MSW, Pr} in the limit $q \to 1$. For type $A$, the proof of the theorem uses work of Tran \cite{T} on quantum $F$-polynomials.

\bigskip

We then get a curious  application of Theorem A by comparing our expansion formula for self-crossing arcs in the Kronecker type with generating functions of BPS states. Relations between BPS states and cluster theory have been studied for example by \cite{ACCERV,ACCERV2,Ci,CN,GMN,W}. C\'{o}rdova--Neitzke \cite{CN} consider a supersymmetric quantum field theory of quiver type whose line defects $W_n$ are parametrized by natural numbers $n$ and compute the generating  function $F(W_n)$ of BPS states recursively using a Coulomb branch formula for small $n$. 

\begin{questIntro}[Question \ref{Conj:BPS}]
Is the following statement true? For $n\in\mathbb{N}$, the generating function $F(W_n)$ coincides with the element associated to a self-crossing arc in the quantum cluster algebra of Kronecker type.
\end{questIntro}

We answer the question affirmatively for $n\in\{0,1,2,3,4\}$, see Proposition~\ref{prop:BPS}.

\bigskip

In an independent work, Huang~\cite{H} provides an expansion formula for quantum cluster algebras with \emph{boundary} coefficients generalizing \cite{MSW}; here the $q$-powers of the terms are given by recursions in the perfect matching lattice.

\bigskip

If we plug in $q=1$ in Theorem A, we obtain the classical expansion formula. The substitution $q=-1$ is related to Stembridge's $q=-1$ phenomenon. The snake graph $\SG$ associated with a (non-crossing) arc in the annulus with exactly one marked point on each boundary component admits an axis of symmetry. Reflection across this axis induces an involution $\sigma$ on the set of perfect matchings $\operatorname{Match}(\SG)$. Stembridge's $q=-1$ phenomenon \cite{S} asserts that the number of fixed points under this involution is equal to the evaluation of coefficients at $q=-1$. We define a map $w$ as a slight variation of the map $\Omega$. We then get the following result.

\begin{thmIntro}[Corollary \ref{Cor:Stembridge}] 
Let $p,r\in\mathbb{N}$. There exists $X(q)\in\mathbb{Z}[q]$ such that the following conditions hold.
\begin{enumerate}
\item Up to a power of $q^{\pm 1/2}$, the polynomial $X(q)$ is equal to $\sum_{P} q^{\Omega (P)}$ where the sum runs over all $P$ in the graded component $\Match(\SG)_{p,r}$ of the perfect matching lattice. 
\item The value $X(-1)$ is equal to the number of fixed points of the restriction of $\sigma$ to $\Match(\SG)_{p,r}$. 
\end{enumerate}
 In other words, the quadruple $(\Match(\SG)_{p,r},\sigma,X(q),w)$ satisfies Stembridge's $q=-1$ phenomenon.
\end{thmIntro}

\bigskip

The paper is organized as follows. Section 2 is devoted to background on surface cluster algebras, snake graphs and general introduction to quantum cluster algebras. Section 3 concerns with the expansion formula for quantum cluster algebras of type $A$. Section 4 gives the formula for the Kronecker type and links it to BPS states.
Finally in Section 5, we establish a connection to Stembridge's $q=-1$ phenomenon.

\bigskip

\emph{Acknowledgements:} We would like to thank Hugh Thomas for stimulating discussions and  for bringing C\'ordova--Neitzke's paper to our attention. We would like to thank Dylan Allegretti for insightful comments on a previous version of the manuscript. 

\section{Background on (quantum) cluster algebras and snake graphs}
\subsection{Cluster algebras}

Let us recall some basic features of cluster algebras. For more details we refer the reader to the literature.

Let $m\geq n$ be natural numbers. In this article we use the notations $[r]=\{\, k\mid 1\leq k\leq r \,\}$ and $[r,s]=\{\, k\mid r\leq k\leq s\,\}$ for $r,s\in\mathbb{N}$. Moreover, we denote by $I$ the $n\times n$ identity matrix and by $0$ the $n\times n$ zero matrix.

An integer $m\times n$ matrix $B$ is called an \emph{exchange matrix} if the submatrix of $B$ on rows and columns $[n]$ is skew-symmetrizable. We denote by $\Ex (m,n)$ the set of all $m\times n$ exchange matrices. A central notion in cluster theory is the mutation $\mu_k\colon \Ex(m,n) \to \Ex(m,n)$ for $k\in[n]$, defined by $B=(b_{ij})\mapsto \mu_k(B)$ with
\begin{align*}
\mu_k(B)_{i,j}=\begin{cases}
-b_{ij},&\textrm{if }i=k\textrm{ or }{j=k};\\
b_{ij}+\left(b_{ik}\lvert b_{kj}\rvert+\lvert b_{ik}\rvert b_{kj}\right)/2,&\textrm{otherwise}.
\end{cases}
\end{align*}
A \emph{cluster} is a tuple $(\x,\y)=(x_1,\ldots,x_n,y_{1},\ldots,y_{m-n})$ of $m$ algebraically independent variables over $\mathbb{Q}$. The elements $x_i$ with $i\in[n]$ are called \emph{cluster variables} and the elements $y_i$ with $i\in[1,m-n]$ are called \emph{frozen variables}. A \emph{seed} is a triple $(B,\x,\y)$ formed by a cluster and an exchange matrix. We extend the notion of mutation to seeds in the following way. The \emph{mutation} of a seed $(B,\x,\y)$ at $k\in[n]$ is the seed $(\mu_k(B),\x',\y)$ where $\x'$ is obtained from $\x$ by replacing $x_k$ with 
\begin{align*}
x_k'=\frac{1}{x_k}\left(\prod_{1\leq i \leq n\atop b_{ik}>0} x_i^{b_{ik}}  \prod_{n+1\leq i \leq m\atop b_{ik}>0} y_{i-n}^{b_{ik}} +\prod_{1\leq i \leq n\atop b_{ik}<0} x_i^{-b_{ik}} \prod_{n+1\leq i \leq m\atop b_{ik}<0} y_{i-n}^{-b_{ik}}\right).
\end{align*}
The \emph{cluster algebra} $\A(B,\x,\y)$ is the subalgebra of $\mathbb{Q}(\x,\y)$ generated by all cluster variables in all seeds obtained from $(B,\x,\y)$ by sequences of mutations.

\begin{defn}[Principal coefficients and $C$-matrices]
\label{Defn:Cmatrix}
\begin{itemize}
\item[(a)] Let $\widetilde{B}$ be an $n\times n$ exchange matrix. The $(2n)\times n$ block matrix 
\begin{align*}
B=\left(\begin{matrix}\widetilde{B}\\I\end{matrix}\right)
\end{align*} 
is said to be the extension of $\widetilde{B}$ with \emph{principal coefficients}. 
\item[(b)] Suppose we mutate $B$ along a sequence $\mathbf{i}=(i_r,\ldots,i_2,i_1)$ of mutable indices to obtain a matrix 
\begin{align*}
\mu_{\mathbf{i}}(B)=B_{\mathbf{i}}=\left(\begin{matrix}\widetilde{B}_{\mathbf{i}}\\C_{\mathbf{i}}\end{matrix}\right).
\end{align*}
Then we call $C_{\mathbf{i}}$ the \emph{$C$-matrix} (or \emph{coefficient matrix}) of $B_{\mathbf{i}}$. The column vectors of $C_{\mathbf{i}}$ are called \emph{$c$-vectors}.  
\end{itemize}
\end{defn}

Let $\widetilde{B}$ be an $n\times n$ skew-symmetrizable matrix and let $B$ be its extension with principal coefficients as above. We choose an initial cluster with cluster variables $\x=(x_1,x_2,\ldots,x_n)$ and frozen variables $\y=(y_1,y_2,\ldots,y_n)$. 

Suppose we mutate $(B,\x,\y)$ along a sequence $\mathbf{i}=(i_r,\ldots,i_2,i_1)$ of mutable indices to obtain a seed $(B_{\mathbf{i}},\x_{\mathbf{i}},\y_{\mathbf{i}})$. Let $x'\in\x_{\mathbf{i}}$. Then we denote the $c$-vector corresponding to the column of $C$ with the same label as $x'$ by $c(x')$.

Fomin--Zelevinsky \cite[Proposition 6.1]{FZ4} prove that the assignments $x_i\mapsto e_i$ and $y_i\mapsto -b_i$, where $e_i$ is the $i$-th standard basis vector and $b_i$ is the $i$-th column vector of $\widetilde{B}$, define a $\mathbb{Z}^n$-grading of the Laurent polynomial ring $\mathbb{Z}[\, x_i^{\pm 1},y_i^{\pm} \mid  i\in [n] \,]$ such that each cluster variable $x'\in\A(B,\x,\y)$ is homogeneous.

\begin{defn}[$G$-matrices] The \emph{$g$-vector} $g(x')$ of a cluster variable $x'\in\A(B,\x,\y)$ is given by the degree of $x'$ with respect to the $\mathbb{Z}^n$-grading. The \emph{$G$-matrix} of a cluster $\x'=(x_1',x_2',\ldots,x_n')$ is defined to be the matrix with columns $g(x_1'),g(x_2'),\ldots,g(x_n').$ 
\end{defn}

Suppose we mutate $(B,\x,\y)$ along a sequence $\mathbf{i}=(i_r,\ldots,i_2,i_1)$ of mutable indices to obtain a seed $(B_{\mathbf{i}},\x_{\mathbf{i}},\y_{\mathbf{i}})$. Nakanishi \cite[Theorem 4.1]{N} and Nakanishi--Zelevinsky \cite[Theorem 1.2]{NZ} prove that the $C$-matrix of $B_{\mathbf{i}}$ and the $G$-matrix of $\x_{\mathbf{i}}$ possess the tropical duality $C_{\mathbf{i}}G_{\mathbf{i}}^T=I$.

\subsection{Surface cluster algebras and snake graphs}

The focus of this article is on cluster algebras of type $A_n$ and of Kronecker type each of which can be realized as a cluster algebra of surface type \cite{FST,FT}. We will recall surface cluster algebras in this section following the exposition in \cite{CS1,MSW} in order to give a self-contained exposition of the paper.

Let $\cals$ be a connected oriented $2$-dimensional Riemann surface (with  possibly empty) boundary and $\calm$ be a finite set of marked points. We further assume that $\calm$ is contained in the boundary of $\cals$ (though this is not necessary for the general construction) and that each boundary component contains at least one marked point. If $\cals$ is a disc, we let $\calm$ to have at least $4$ marked points. We call a pair $(\cals,\calm)$ \emph{bordered surface with marked points}.

\begin{defn} [Arcs, compatible arcs, triangulation] Let~$\cals$ be a bordered surface with marked points. An \emph{arc} of $\cals$ is a non-self-intersecting curve, considered up to isotopy, with endpoints at marked points of $\cals$. We further assume that an arc is disjoint from the boundary of $\cals$ except the endpoints, which might coincide, and it does not cut out a monogon or digon. Two arcs are \emph{compatible} if they do not intersect in the interior of $\cals$, that is, there are representatives in the corresponding isotopy classes that do not intersect except the endpoints may coincide. A \emph{triangulation} $\calt$ is a maximal collection of pairwise distinct compatible arcs. The connected components in the complement $\cals\backslash \calt$ are called \emph{triangles}.
\end{defn}

Let~$\cals$ be a bordered surface with marked points and let $\calt=\{\tau_1,\dots,\tau_n\}$ be a triangulation of $\cals$. Every triangle $\Delta$ is bounded by three arcs $\tau_r$, $\tau_s$, $\tau_t$. Without loss of generality we may assume that, when going through the boundary of $\Delta$ in positive direction, we pass the arcs in order $\tau_r\to \tau_s\to\tau_t\to\tau_r$. To every triangle $\Delta$ we associate the $n\times n$ matrix $B^{\Delta}=(b_{ij}^{\Delta})$ with entries $b_{rs}=b_{st}=b_{tr}=1$, $b_{sr}=b_{ts}=b_{rt}=-1$, and $b_{ij}=0$ otherwise.

\begin{defn}[Surface cluster algebras] Let~$\cals$ be a bordered surface with marked points and $\calt=\{\tau_1,\dots,\tau_n\}$ be a triangulation of $\cals$. 
\begin{itemize}
\item For each $\tau_i\in\calt$ with $i\in[n]$ we associate formal variables $x_i$ and $y_i$. Then $(\x,\y)=(x_1,\dots,x_n,y_1,\ldots,y_n)$ is an initial cluster.
\item The \emph{signed adjacency matrix} of $\calt$ is $\widetilde{B}=\sum_{\Delta}B^{\Delta}$ where the sum runs over all triangles. The \emph{adjacency quiver} $Q_{\calt}$ of $(\cals,\calt)$ is a quiver with vertex set $[n]$ such that the number of arrows  $i\to j$ is equal to $b_{ij}$ if $b_{ij}>0$ and zero otherwise.
\item The \emph{surface cluster algebra} $\cala(\cals,\calt)$ with principal coefficients is the cluster algebra generated by the seed $(B,\x,\y)$ where $B$ is the principal extension of $\widetilde{B}$.
\end{itemize} 
\end{defn}

\subsubsection{Snake graphs} 
In this section, we recall the construction of snake graphs of \cite{MSW}. We fix a surface $\cals$ and a triangulation $\calt=\{\tau_1,\tau_2,\dots,\tau_n\}$.

Let $\gamma$ be a (generalized) arc in $(\cals,\calt).$ We consider the domain $\operatorname{dom}(\gamma)$ of $\gamma$ in $(\cals,\calt)$, that is, $\operatorname{dom}(\gamma)$ is the union of all the triangles $\gamma$ crosses (counted with multiplicity), see Figure~\ref{Fig:Dom_SG}. 

Fix an orientation of $\gamma.$ For each crossing of $\gamma$ with $\calt,$ we will associate a weighted tile, i.e. a square of fixed side length drawn in the plane with sides aligned horizontally and vertically, as follows:

\begin{itemize}
\item if $\gamma$ crosses an arc $\tau_i,$ the associated weighted tile $G_{i}$ has face weight $w(G_i)=\tau_i$ and edge weights induced by the quadrilateral $Q_i$ in $\calt$ which contains $\tau_i$ as diagonal, see Figure~\ref{Fig:Dom_SG};
\item glue successive tiles $G_{i_1},G_{i_2},\dots,G_{i_k}$ as follows: set the diagonals $\tau_k$ and $\tau_{k+1}$ from the top-left corner to the bottom-right corner of $G_k$ and $G_{k+1},$ respectively and glue the tiles $G_k$ and $G_{k+1}$ along the edge with weight induced from the third side of the triangle bounded by $\tau_k$ and $\tau_{k+1}$ in $\calt,$ see Figure~\ref{Fig:Dom_SG}.  
\end{itemize}

\begin{figure}
\begin{tikzpicture}
\node[scale=.8] at (0,0){$\begin{tikzpicture}[xscale=4,yscale=2]  

    \draw (0,0) ellipse (1 and 1);
    
    \newcommand{\DegOne}{320}  
    \node at ($(\DegOne:1 and 1)$) {$\bullet$};
    
    \newcommand{\DegTwo}{280}  
    \node at ($(\DegTwo:1 and 1)$) {$\bullet$};
    
    \newcommand{\DegThree}{90}  
    \node at ($(\DegThree:1 and 1)$) {$\bullet$};    
    
    \newcommand{\DegFour}{140}  
    \node at ($(\DegFour:1 and 1)$) {$\bullet$};
    
    \newcommand{\DegFive}{210}  
    \node at ($(\DegFive:1 and 1)$) {$\bullet$};
    
    \newcommand{\DegSix}{230}  
    \node at ($(\DegSix:1 and 1)$) {$\bullet$};  
    
    \newcommand{\DegSeven}{250}  
    \node at ($(\DegSeven:1 and 1)$) {$\bullet$};
    
    \draw ($(\DegFour:1 and 1)$) to ($(\DegFive:1 and 1)$);
    \draw ($(\DegFour:1 and 1)$) to ($(\DegSix:1 and 1)$); 
    \draw ($(\DegFour:1 and 1)$) edge ["\small{$\tau_{i-1}$}"] ($(\DegSeven:1 and 1)$);  
    
    \draw ($(\DegThree:1 and 1)$) edge ["\small{$\tau_{i+2}$}"] ($(\DegOne:1 and 1)$); 
    \draw ($(\DegThree:1 and 1)$) edge ["\small{$\tau_{i+1}$}"] ($(\DegTwo:1 and 1)$); 
    \draw ($(\DegThree:1 and 1)$) edge ["\small{$\tau_i$}"] ($(\DegSeven:1 and 1)$);

\draw[color=white] ($(\DegSeven:1 and 1)$) -- ($(\DegTwo:1 and 1)$) node [midway,above,color=black]{\small{a}};    

\draw[color=white] ($(\DegFour:1 and 1)$) -- ($(\DegThree:1 and 1)$) node [midway,color=black]{\small{b}};    
   
    \node at ($(0:1 and 1)$) {\textcolor{red}{$\bullet$}};  
    \node at ($(180:1 and 1)$) {\textcolor{red}{$\bullet$}}; 
    \draw[color=red,bend right=30] ($(0:1 and 1)$) to ($(180:1 and 1)$);
  
\end{tikzpicture}$};
\node[scale=.8] at (7,0){$\begin{tikzpicture}

\newcommand{\sidelength}{1.8}

\node (A)  at (0,0) {$\bullet$};
\node (B) at (\sidelength,0) {$\bullet$};
\node (C) at (2*\sidelength,0) {$\bullet$};

\node (D) at (0,\sidelength) {$\bullet$};
\node (E) at (\sidelength,\sidelength) {$\bullet$};
\node (F) at (2*\sidelength,\sidelength) {$\bullet$};  

\node (1)  at (-.5*\sidelength,.5*\sidelength) {$\dots$};
\node (2) at (2.7*\sidelength,.5*\sidelength) {$\dots$};  

\draw (0,0) -- (\sidelength,0) node [midway,below] {$\tau_{i-1}$};
\draw (0,\sidelength) -- (\sidelength,0) node [midway,right] {$\tau_{i}$};
\draw (0,0) -- (0,\sidelength) node[midway,left]{$b$};
\draw (\sidelength,0) -- (\sidelength,\sidelength) node[midway,right] {$a$};
\draw (0,\sidelength) -- (\sidelength,\sidelength) node[midway,above] {$\tau_{i+1}$};

\draw (\sidelength,\sidelength) -- (2*\sidelength,0) node[midway,right] {$\tau_{i+1}$};
\draw (\sidelength,0) -- (2*\sidelength,0) node [midway,below] {$\tau_{i}$};
\draw (2*\sidelength,0) -- (2*\sidelength,\sidelength) node [midway,right] {$\tau_{i+2}$};
\draw (\sidelength,\sidelength) -- (2*\sidelength,\sidelength);    
    
\end{tikzpicture}$};
\end{tikzpicture}
\caption{The domain of an arc (left) and the construction of its snake graph (right)}
\label{Fig:Dom_SG}
\end{figure}
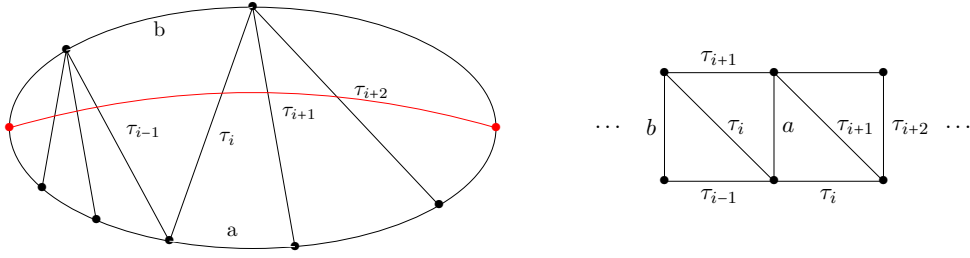

\begin{rem}
 The snake graph $\calg_{\gamma}$ associated to an arc $\gamma$ in a triangulated surface is simply an unfolding of the domain of $\gamma$ where zig-zag pieces in the triangulation correspond to straight subgraphs in $G_{\gamma}$ and fans in the triangulation correspond to corner subgraphs in $G_{\gamma}$. The edge and face weights on $\calg_{\gamma}$ are induced by the triangulation $\calt.$
\end{rem}

\begin{defn} A \emph{perfect matching} of a graph $\SG$ is a set of edges of $\SG$ such that each vertex $v$ in $\SG$ is incident to precisely one edge $e$ in $P.$ We denote by $\Match (\SG)$ the set of perfect matchings of the graph $\SG$, and for brevity we write 
$\PM \models \SG$ if $\PM$ is a perfect matching of $\SG$.
\end{defn}

The set $\Match(\SG_{\gamma})$ of a snake graph $\SG_{\gamma}$ admits a rich combinatorial structure which we will discuss next.

\begin{defn}[Perfect matching lattice]
The \emph{perfect matching graph} $L(\SG)$ of a snake graph $\SG$ is defined as follows: its vertices are the perfect matchings of $\SG$, and two vertices are connected by an edge if the perfect matchings are obtained from each other by a single twist.
\end{defn}

By \cite{MSW}, the perfect matching graph admits a lattice structure with maximal and minimal perfect matchings given by those  containing only boundary edges of $\SG$. We refer to \cite{MSW} for a precise definition of the minimum and maximum perfect matchings and explain the lattice later in more detail.

\begin{ex} The minimal and maximal perfect matchings of the snake graph \begin{tikzpicture}[scale=.7]
\newcommand{\hd}{3cm} 
\newcommand{\vd}{1.8cm} 
\newcommand{\dist}{0.5cm} 

\path[draw] (1*\dist,0) edge node {} (2*\dist,0);
\path[draw] (2*\dist,0) edge node {} (3*\dist,0);
\path[draw] (3*\dist,0) edge node {} (4*\dist,0);
\path[draw] (4*\dist,0) edge node {} (5*\dist,0);

\path[draw] (1*\dist,\dist) edge node {} (2*\dist,\dist);
\path[draw] (2*\dist,\dist) edge node {} (3*\dist,\dist);
\path[draw] (3*\dist,\dist) edge node {} (4*\dist,\dist);
\path[draw] (4*\dist,\dist) edge node {} (5*\dist,\dist);

\path[draw] (4*\dist,2*\dist) edge node {} (5*\dist,2*\dist);

\path[draw] (1*\dist,0) edge node {} (1*\dist,\dist);
\path[draw] (2*\dist,0) edge node {} (2*\dist,\dist);
\path[draw] (3*\dist,0) edge node {} (3*\dist,\dist);
\path[draw] (4*\dist,0) edge node {} (4*\dist,\dist);
\path[draw] (5*\dist,0) edge node {} (5*\dist,\dist);
\path[draw] (4*\dist,\dist) edge node {} (4*\dist,2*\dist);
\path[draw] (5*\dist,\dist) edge node {} (5*\dist,2*\dist);

\node[scale=.8] (T1) at (1.5*\dist,0.5*\dist) {$3$};
\node[scale=.8] (T2) at (2.5*\dist,0.5*\dist) {$1$};
\node[scale=.8] (T3) at (3.5*\dist,0.5*\dist) {$2$};
\node[scale=.8] (T4) at (4.5*\dist,0.5*\dist) {$3$};
\node[scale=.8] (T5) at (4.5*\dist,1.5*\dist) {$1$};

\end{tikzpicture} are \begin{tikzpicture}[scale=.7]
\newcommand{\hd}{3cm} 
\newcommand{\vd}{1.8cm} 
\newcommand{\dist}{0.5cm} 

\path[draw] (1*\dist,0) edge node {} (2*\dist,0);
\path[draw,ultra thick] (2*\dist,0) edge node {} (3*\dist,0);
\path[draw] (3*\dist,0) edge node {} (4*\dist,0);
\path[draw,ultra thick] (4*\dist,0) edge node {} (5*\dist,0);

\path[draw] (1*\dist,\dist) edge node {} (2*\dist,\dist);
\path[draw,ultra thick] (2*\dist,\dist) edge node {} (3*\dist,\dist);
\path[draw] (3*\dist,\dist) edge node {} (4*\dist,\dist);
\path[draw] (4*\dist,\dist) edge node {} (5*\dist,\dist);

\path[draw] (4*\dist,2*\dist) edge node {} (5*\dist,2*\dist);

\path[draw,ultra thick] (1*\dist,0) edge node {} (1*\dist,\dist);
\path[draw] (2*\dist,0) edge node {} (2*\dist,\dist);
\path[draw] (3*\dist,0) edge node {} (3*\dist,\dist);
\path[draw] (4*\dist,0) edge node {} (4*\dist,\dist);
\path[draw] (5*\dist,0) edge node {} (5*\dist,\dist);
\path[draw,ultra thick] (4*\dist,\dist) edge node {} (4*\dist,2*\dist);
\path[draw,ultra thick] (5*\dist,\dist) edge node {} (5*\dist,2*\dist);

\node[scale=.8] (T1) at (1.5*\dist,0.5*\dist) {$3$};
\node[scale=.8] (T2) at (2.5*\dist,0.5*\dist) {$1$};
\node[scale=.8] (T3) at (3.5*\dist,0.5*\dist) {$2$};
\node[scale=.8] (T4) at (4.5*\dist,0.5*\dist) {$3$};
\node[scale=.8] (T5) at (4.5*\dist,1.5*\dist) {$1$};

\end{tikzpicture} and \begin{tikzpicture}[scale=.7]
\newcommand{\hd}{3cm} 
\newcommand{\vd}{1.8cm} 
\newcommand{\dist}{0.5cm} 

\path[draw,ultra thick] (1*\dist,0) edge node {} (2*\dist,0);
\path[draw] (2*\dist,0) edge node {} (3*\dist,0);
\path[draw,ultra thick] (3*\dist,0) edge node {} (4*\dist,0);
\path[draw] (4*\dist,0) edge node {} (5*\dist,0);

\path[draw,ultra thick] (1*\dist,\dist) edge node {} (2*\dist,\dist);
\path[draw] (2*\dist,\dist) edge node {} (3*\dist,\dist);
\path[draw,ultra thick] (3*\dist,\dist) edge node {} (4*\dist,\dist);
\path[draw] (4*\dist,\dist) edge node {} (5*\dist,\dist);

\path[draw,ultra thick] (4*\dist,2*\dist) edge node {} (5*\dist,2*\dist);

\path[draw] (1*\dist,0) edge node {} (1*\dist,\dist);
\path[draw] (2*\dist,0) edge node {} (2*\dist,\dist);
\path[draw] (3*\dist,0) edge node {} (3*\dist,\dist);
\path[draw] (4*\dist,0) edge node {} (4*\dist,\dist);
\path[draw,ultra thick] (5*\dist,0) edge node {} (5*\dist,\dist);

\path[draw] (4*\dist,\dist) edge node {} (4*\dist,2*\dist);
\path[draw] (5*\dist,\dist) edge node {} (5*\dist,2*\dist);

\node[scale=.8] (T1) at (1.5*\dist,0.5*\dist) {$3$};
\node[scale=.8] (T2) at (2.5*\dist,0.5*\dist) {$1$};
\node[scale=.8] (T3) at (3.5*\dist,0.5*\dist) {$2$};
\node[scale=.8] (T4) at (4.5*\dist,0.5*\dist) {$3$};
\node[scale=.8] (T5) at (4.5*\dist,1.5*\dist) {$1$};

\end{tikzpicture}
and its perfect matching lattice is given in Figure~\ref{Figure:PMLattice}.

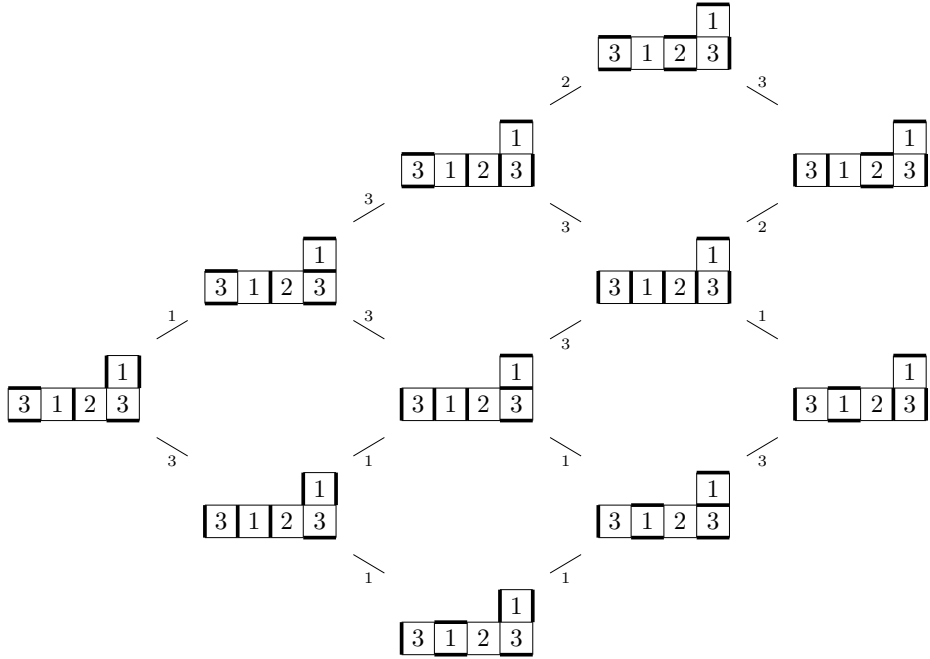
\begin{figure}

\begin{center}

\resizebox{.8\textwidth}{!}{
\begin{tikzpicture}

\newcommand{\hd}{3cm} 
\newcommand{\vd}{1.8cm} 
\newcommand{\dist}{0.5cm} 

\node (0) at (0,0) {\begin{tikzpicture}

\path[draw] (1*\dist,0) edge node {} (2*\dist,0);
\path[draw,ultra thick] (2*\dist,0) edge node {} (3*\dist,0);
\path[draw] (3*\dist,0) edge node {} (4*\dist,0);
\path[draw,ultra thick] (4*\dist,0) edge node {} (5*\dist,0);

\path[draw] (1*\dist,\dist) edge node {} (2*\dist,\dist);
\path[draw,ultra thick] (2*\dist,\dist) edge node {} (3*\dist,\dist);
\path[draw] (3*\dist,\dist) edge node {} (4*\dist,\dist);
\path[draw] (4*\dist,\dist) edge node {} (5*\dist,\dist);

\path[draw] (4*\dist,2*\dist) edge node {} (5*\dist,2*\dist);

\path[draw,ultra thick] (1*\dist,0) edge node {} (1*\dist,\dist);
\path[draw] (2*\dist,0) edge node {} (2*\dist,\dist);
\path[draw] (3*\dist,0) edge node {} (3*\dist,\dist);
\path[draw] (4*\dist,0) edge node {} (4*\dist,\dist);
\path[draw] (5*\dist,0) edge node {} (5*\dist,\dist);
\path[draw,ultra thick] (4*\dist,\dist) edge node {} (4*\dist,2*\dist);
\path[draw,ultra thick] (5*\dist,\dist) edge node {} (5*\dist,2*\dist);

\node (T1) at (1.5*\dist,0.5*\dist) {$3$};
\node (T2) at (2.5*\dist,0.5*\dist) {$1$};
\node (T3) at (3.5*\dist,0.5*\dist) {$2$};
\node (T4) at (4.5*\dist,0.5*\dist) {$3$};
\node (T5) at (4.5*\dist,1.5*\dist) {$1$};

\end{tikzpicture}};

\node (1) at (-\hd,\vd) {\begin{tikzpicture}

\path[draw] (1*\dist,0) edge node {} (2*\dist,0);
\path[draw] (2*\dist,0) edge node {} (3*\dist,0);
\path[draw] (3*\dist,0) edge node {} (4*\dist,0);
\path[draw,ultra thick] (4*\dist,0) edge node {} (5*\dist,0);

\path[draw] (1*\dist,\dist) edge node {} (2*\dist,\dist);
\path[draw] (2*\dist,\dist) edge node {} (3*\dist,\dist);
\path[draw] (3*\dist,\dist) edge node {} (4*\dist,\dist);
\path[draw] (4*\dist,\dist) edge node {} (5*\dist,\dist);

\path[draw] (4*\dist,2*\dist) edge node {} (5*\dist,2*\dist);

\path[draw,ultra thick] (1*\dist,0) edge node {} (1*\dist,\dist);
\path[draw,ultra thick] (2*\dist,0) edge node {} (2*\dist,\dist);
\path[draw,ultra thick] (3*\dist,0) edge node {} (3*\dist,\dist);
\path[draw] (4*\dist,0) edge node {} (4*\dist,\dist);
\path[draw] (5*\dist,0) edge node {} (5*\dist,\dist);
\path[draw,ultra thick] (4*\dist,\dist) edge node {} (4*\dist,2*\dist);
\path[draw,ultra thick] (5*\dist,\dist) edge node {} (5*\dist,2*\dist);

\node (T1) at (1.5*\dist,0.5*\dist) {$3$};
\node (T2) at (2.5*\dist,0.5*\dist) {$1$};
\node (T3) at (3.5*\dist,0.5*\dist) {$2$};
\node (T4) at (4.5*\dist,0.5*\dist) {$3$};
\node (T5) at (4.5*\dist,1.5*\dist) {$1$};

\end{tikzpicture}};

\node (2) at (\hd,\vd) {\begin{tikzpicture}

\path[draw] (1*\dist,0) edge node {} (2*\dist,0);
\path[draw,ultra thick] (2*\dist,0) edge node {} (3*\dist,0);
\path[draw] (3*\dist,0) edge node {} (4*\dist,0);
\path[draw,ultra thick] (4*\dist,0) edge node {} (5*\dist,0);

\path[draw] (1*\dist,\dist) edge node {} (2*\dist,\dist);
\path[draw,ultra thick] (2*\dist,\dist) edge node {} (3*\dist,\dist);
\path[draw] (3*\dist,\dist) edge node {} (4*\dist,\dist);
\path[draw,ultra thick] (4*\dist,\dist) edge node {} (5*\dist,\dist);

\path[draw,ultra thick] (4*\dist,2*\dist) edge node {} (5*\dist,2*\dist);

\path[draw,ultra thick] (1*\dist,0) edge node {} (1*\dist,\dist);
\path[draw] (2*\dist,0) edge node {} (2*\dist,\dist);
\path[draw] (3*\dist,0) edge node {} (3*\dist,\dist);
\path[draw] (4*\dist,0) edge node {} (4*\dist,\dist);
\path[draw] (5*\dist,0) edge node {} (5*\dist,\dist);
\path[draw] (4*\dist,\dist) edge node {} (4*\dist,2*\dist);
\path[draw] (5*\dist,\dist) edge node {} (5*\dist,2*\dist);

\node (T1) at (1.5*\dist,0.5*\dist) {$3$};
\node (T2) at (2.5*\dist,0.5*\dist) {$1$};
\node (T3) at (3.5*\dist,0.5*\dist) {$2$};
\node (T4) at (4.5*\dist,0.5*\dist) {$3$};
\node (T5) at (4.5*\dist,1.5*\dist) {$1$};

\end{tikzpicture}};

\node (3) at (0,2*\vd) {\begin{tikzpicture}

\path[draw] (1*\dist,0) edge node {} (2*\dist,0);
\path[draw] (2*\dist,0) edge node {} (3*\dist,0);
\path[draw] (3*\dist,0) edge node {} (4*\dist,0);
\path[draw,ultra thick] (4*\dist,0) edge node {} (5*\dist,0);

\path[draw] (1*\dist,\dist) edge node {} (2*\dist,\dist);
\path[draw] (2*\dist,\dist) edge node {} (3*\dist,\dist);
\path[draw] (3*\dist,\dist) edge node {} (4*\dist,\dist);
\path[draw,ultra thick] (4*\dist,\dist) edge node {} (5*\dist,\dist);

\path[draw,ultra thick] (4*\dist,2*\dist) edge node {} (5*\dist,2*\dist);

\path[draw,ultra thick] (1*\dist,0) edge node {} (1*\dist,\dist);
\path[draw,ultra thick] (2*\dist,0) edge node {} (2*\dist,\dist);
\path[draw,ultra thick] (3*\dist,0) edge node {} (3*\dist,\dist);
\path[draw] (4*\dist,0) edge node {} (4*\dist,\dist);
\path[draw] (5*\dist,0) edge node {} (5*\dist,\dist);
\path[draw] (4*\dist,\dist) edge node {} (4*\dist,2*\dist);
\path[draw] (5*\dist,\dist) edge node {} (5*\dist,2*\dist);

\node (T1) at (1.5*\dist,0.5*\dist) {$3$};
\node (T2) at (2.5*\dist,0.5*\dist) {$1$};
\node (T3) at (3.5*\dist,0.5*\dist) {$2$};
\node (T4) at (4.5*\dist,0.5*\dist) {$3$};
\node (T5) at (4.5*\dist,1.5*\dist) {$1$};

\end{tikzpicture}};

\node (4) at (-2*\hd,2*\vd) {\begin{tikzpicture}

\path[draw,ultra thick] (1*\dist,0) edge node {} (2*\dist,0);
\path[draw] (2*\dist,0) edge node {} (3*\dist,0);
\path[draw] (3*\dist,0) edge node {} (4*\dist,0);
\path[draw,ultra thick] (4*\dist,0) edge node {} (5*\dist,0);

\path[draw,ultra thick] (1*\dist,\dist) edge node {} (2*\dist,\dist);
\path[draw] (2*\dist,\dist) edge node {} (3*\dist,\dist);
\path[draw] (3*\dist,\dist) edge node {} (4*\dist,\dist);
\path[draw] (4*\dist,\dist) edge node {} (5*\dist,\dist);

\path[draw] (4*\dist,2*\dist) edge node {} (5*\dist,2*\dist);

\path[draw] (1*\dist,0) edge node {} (1*\dist,\dist);
\path[draw] (2*\dist,0) edge node {} (2*\dist,\dist);
\path[draw,ultra thick] (3*\dist,0) edge node {} (3*\dist,\dist);
\path[draw] (4*\dist,0) edge node {} (4*\dist,\dist);
\path[draw] (5*\dist,0) edge node {} (5*\dist,\dist);
\path[draw,ultra thick] (4*\dist,\dist) edge node {} (4*\dist,2*\dist);
\path[draw,ultra thick] (5*\dist,\dist) edge node {} (5*\dist,2*\dist);

\node (T1) at (1.5*\dist,0.5*\dist) {$3$};
\node (T2) at (2.5*\dist,0.5*\dist) {$1$};
\node (T3) at (3.5*\dist,0.5*\dist) {$2$};
\node (T4) at (4.5*\dist,0.5*\dist) {$3$};
\node (T5) at (4.5*\dist,1.5*\dist) {$1$};

\end{tikzpicture}};

\node (5) at (2*\hd,2*\vd) {\begin{tikzpicture}

\path[draw] (1*\dist,0) edge node {} (2*\dist,0);
\path[draw,ultra thick] (2*\dist,0) edge node {} (3*\dist,0);
\path[draw] (3*\dist,0) edge node {} (4*\dist,0);
\path[draw] (4*\dist,0) edge node {} (5*\dist,0);

\path[draw] (1*\dist,\dist) edge node {} (2*\dist,\dist);
\path[draw,ultra thick] (2*\dist,\dist) edge node {} (3*\dist,\dist);
\path[draw] (3*\dist,\dist) edge node {} (4*\dist,\dist);
\path[draw] (4*\dist,\dist) edge node {} (5*\dist,\dist);

\path[draw,ultra thick] (4*\dist,2*\dist) edge node {} (5*\dist,2*\dist);

\path[draw,ultra thick] (1*\dist,0) edge node {} (1*\dist,\dist);
\path[draw] (2*\dist,0) edge node {} (2*\dist,\dist);
\path[draw] (3*\dist,0) edge node {} (3*\dist,\dist);
\path[draw,ultra thick] (4*\dist,0) edge node {} (4*\dist,\dist);
\path[draw,ultra thick] (5*\dist,0) edge node {} (5*\dist,\dist);

\path[draw] (4*\dist,\dist) edge node {} (4*\dist,2*\dist);
\path[draw] (5*\dist,\dist) edge node {} (5*\dist,2*\dist);

\node (T1) at (1.5*\dist,0.5*\dist) {$3$};
\node (T2) at (2.5*\dist,0.5*\dist) {$1$};
\node (T3) at (3.5*\dist,0.5*\dist) {$2$};
\node (T4) at (4.5*\dist,0.5*\dist) {$3$};
\node (T5) at (4.5*\dist,1.5*\dist) {$1$};

\end{tikzpicture}};

\node (6) at (-\hd,3*\vd) {\begin{tikzpicture}

\path[draw,ultra thick] (1*\dist,0) edge node {} (2*\dist,0);
\path[draw] (2*\dist,0) edge node {} (3*\dist,0);
\path[draw] (3*\dist,0) edge node {} (4*\dist,0);
\path[draw,ultra thick] (4*\dist,0) edge node {} (5*\dist,0);

\path[draw,ultra thick] (1*\dist,\dist) edge node {} (2*\dist,\dist);
\path[draw] (2*\dist,\dist) edge node {} (3*\dist,\dist);
\path[draw] (3*\dist,\dist) edge node {} (4*\dist,\dist);
\path[draw,ultra thick] (4*\dist,\dist) edge node {} (5*\dist,\dist);

\path[draw,ultra thick] (4*\dist,2*\dist) edge node {} (5*\dist,2*\dist);

\path[draw] (1*\dist,0) edge node {} (1*\dist,\dist);
\path[draw] (2*\dist,0) edge node {} (2*\dist,\dist);
\path[draw,ultra thick] (3*\dist,0) edge node {} (3*\dist,\dist);
\path[draw] (4*\dist,0) edge node {} (4*\dist,\dist);
\path[draw] (5*\dist,0) edge node {} (5*\dist,\dist);
\path[draw] (4*\dist,\dist) edge node {} (4*\dist,2*\dist);
\path[draw] (5*\dist,\dist) edge node {} (5*\dist,2*\dist);

\node (T1) at (1.5*\dist,0.5*\dist) {$3$};
\node (T2) at (2.5*\dist,0.5*\dist) {$1$};
\node (T3) at (3.5*\dist,0.5*\dist) {$2$};
\node (T4) at (4.5*\dist,0.5*\dist) {$3$};
\node (T5) at (4.5*\dist,1.5*\dist) {$1$};

\end{tikzpicture}};

\node (7) at (\hd,3*\vd) {\begin{tikzpicture}

\path[draw] (1*\dist,0) edge node {} (2*\dist,0);
\path[draw] (2*\dist,0) edge node {} (3*\dist,0);
\path[draw] (3*\dist,0) edge node {} (4*\dist,0);
\path[draw] (4*\dist,0) edge node {} (5*\dist,0);

\path[draw] (1*\dist,\dist) edge node {} (2*\dist,\dist);
\path[draw] (2*\dist,\dist) edge node {} (3*\dist,\dist);
\path[draw] (3*\dist,\dist) edge node {} (4*\dist,\dist);
\path[draw] (4*\dist,\dist) edge node {} (5*\dist,\dist);

\path[draw,ultra thick] (4*\dist,2*\dist) edge node {} (5*\dist,2*\dist);

\path[draw,ultra thick] (1*\dist,0) edge node {} (1*\dist,\dist);
\path[draw,ultra thick] (2*\dist,0) edge node {} (2*\dist,\dist);
\path[draw,ultra thick] (3*\dist,0) edge node {} (3*\dist,\dist);
\path[draw,ultra thick] (4*\dist,0) edge node {} (4*\dist,\dist);
\path[draw,ultra thick] (5*\dist,0) edge node {} (5*\dist,\dist);

\path[draw] (4*\dist,\dist) edge node {} (4*\dist,2*\dist);
\path[draw] (5*\dist,\dist) edge node {} (5*\dist,2*\dist);

\node (T1) at (1.5*\dist,0.5*\dist) {$3$};
\node (T2) at (2.5*\dist,0.5*\dist) {$1$};
\node (T3) at (3.5*\dist,0.5*\dist) {$2$};
\node (T4) at (4.5*\dist,0.5*\dist) {$3$};
\node (T5) at (4.5*\dist,1.5*\dist) {$1$};

\end{tikzpicture}};

\node (8) at (0*\hd,4*\vd) {\begin{tikzpicture}

\path[draw,ultra thick] (1*\dist,0) edge node {} (2*\dist,0);
\path[draw] (2*\dist,0) edge node {} (3*\dist,0);
\path[draw] (3*\dist,0) edge node {} (4*\dist,0);
\path[draw] (4*\dist,0) edge node {} (5*\dist,0);

\path[draw,ultra thick] (1*\dist,\dist) edge node {} (2*\dist,\dist);
\path[draw] (2*\dist,\dist) edge node {} (3*\dist,\dist);
\path[draw] (3*\dist,\dist) edge node {} (4*\dist,\dist);
\path[draw] (4*\dist,\dist) edge node {} (5*\dist,\dist);

\path[draw,ultra thick] (4*\dist,2*\dist) edge node {} (5*\dist,2*\dist);

\path[draw] (1*\dist,0) edge node {} (1*\dist,\dist);
\path[draw] (2*\dist,0) edge node {} (2*\dist,\dist);
\path[draw,ultra thick] (3*\dist,0) edge node {} (3*\dist,\dist);
\path[draw,ultra thick] (4*\dist,0) edge node {} (4*\dist,\dist);
\path[draw,ultra thick] (5*\dist,0) edge node {} (5*\dist,\dist);

\path[draw] (4*\dist,\dist) edge node {} (4*\dist,2*\dist);
\path[draw] (5*\dist,\dist) edge node {} (5*\dist,2*\dist);

\node (T1) at (1.5*\dist,0.5*\dist) {$3$};
\node (T2) at (2.5*\dist,0.5*\dist) {$1$};
\node (T3) at (3.5*\dist,0.5*\dist) {$2$};
\node (T4) at (4.5*\dist,0.5*\dist) {$3$};
\node (T5) at (4.5*\dist,1.5*\dist) {$1$};

\end{tikzpicture}};

\node (9) at (2*\hd,4*\vd) {\begin{tikzpicture}

\path[draw] (1*\dist,0) edge node {} (2*\dist,0);
\path[draw] (2*\dist,0) edge node {} (3*\dist,0);
\path[draw,ultra thick] (3*\dist,0) edge node {} (4*\dist,0);
\path[draw] (4*\dist,0) edge node {} (5*\dist,0);

\path[draw] (1*\dist,\dist) edge node {} (2*\dist,\dist);
\path[draw] (2*\dist,\dist) edge node {} (3*\dist,\dist);
\path[draw,ultra thick] (3*\dist,\dist) edge node {} (4*\dist,\dist);
\path[draw] (4*\dist,\dist) edge node {} (5*\dist,\dist);

\path[draw,ultra thick] (4*\dist,2*\dist) edge node {} (5*\dist,2*\dist);

\path[draw,ultra thick] (1*\dist,0) edge node {} (1*\dist,\dist);
\path[draw,ultra thick] (2*\dist,0) edge node {} (2*\dist,\dist);
\path[draw] (3*\dist,0) edge node {} (3*\dist,\dist);
\path[draw] (4*\dist,0) edge node {} (4*\dist,\dist);
\path[draw,ultra thick] (5*\dist,0) edge node {} (5*\dist,\dist);

\path[draw] (4*\dist,\dist) edge node {} (4*\dist,2*\dist);
\path[draw] (5*\dist,\dist) edge node {} (5*\dist,2*\dist);

\node (T1) at (1.5*\dist,0.5*\dist) {$3$};
\node (T2) at (2.5*\dist,0.5*\dist) {$1$};
\node (T3) at (3.5*\dist,0.5*\dist) {$2$};
\node (T4) at (4.5*\dist,0.5*\dist) {$3$};
\node (T5) at (4.5*\dist,1.5*\dist) {$1$};

\end{tikzpicture}};

\node (10) at (\hd,5*\vd) {\begin{tikzpicture}

\path[draw,ultra thick] (1*\dist,0) edge node {} (2*\dist,0);
\path[draw] (2*\dist,0) edge node {} (3*\dist,0);
\path[draw,ultra thick] (3*\dist,0) edge node {} (4*\dist,0);
\path[draw] (4*\dist,0) edge node {} (5*\dist,0);

\path[draw,ultra thick] (1*\dist,\dist) edge node {} (2*\dist,\dist);
\path[draw] (2*\dist,\dist) edge node {} (3*\dist,\dist);
\path[draw,ultra thick] (3*\dist,\dist) edge node {} (4*\dist,\dist);
\path[draw] (4*\dist,\dist) edge node {} (5*\dist,\dist);

\path[draw,ultra thick] (4*\dist,2*\dist) edge node {} (5*\dist,2*\dist);

\path[draw] (1*\dist,0) edge node {} (1*\dist,\dist);
\path[draw] (2*\dist,0) edge node {} (2*\dist,\dist);
\path[draw] (3*\dist,0) edge node {} (3*\dist,\dist);
\path[draw] (4*\dist,0) edge node {} (4*\dist,\dist);
\path[draw,ultra thick] (5*\dist,0) edge node {} (5*\dist,\dist);

\path[draw] (4*\dist,\dist) edge node {} (4*\dist,2*\dist);
\path[draw] (5*\dist,\dist) edge node {} (5*\dist,2*\dist);

\node (T1) at (1.5*\dist,0.5*\dist) {$3$};
\node (T2) at (2.5*\dist,0.5*\dist) {$1$};
\node (T3) at (3.5*\dist,0.5*\dist) {$2$};
\node (T4) at (4.5*\dist,0.5*\dist) {$3$};
\node (T5) at (4.5*\dist,1.5*\dist) {$1$};

\end{tikzpicture}};

\path[draw,-,left] (0) edge node[midway, below] {\tiny{$1$}} (1);
\path[draw,-,right] (0) edge node[midway, below] {\tiny{$1$}} (2);
\path[draw,-,left] (1) edge node[midway, below] {\tiny{$1$}} (3);
\path[draw,-,right] (2) edge node[midway, below] {\tiny{$1$}} (3);
\path[draw,-,left] (1) edge node[midway, below] {\tiny{$3$}} (4);
\path[draw,-,right] (2) edge node[midway, below] {\tiny{$3$}} (5);
\path[draw,-,right] (3) edge node[midway, above] {\tiny{$3$}} (6);
\path[draw,-,left] (4) edge node[midway, above] {\tiny{$1$}} (6);
\path[draw,-,left] (3) edge node[midway, below] {\tiny{$3$}} (7);
\path[draw,-,right] (5) edge node[midway, above] {\tiny{$1$}} (7);
\path[draw,-,left] (6) edge node[midway, above] {\tiny{$3$}} (8);
\path[draw,-,right] (7) edge node[midway, below] {\tiny{$3$}} (8);
\path[draw,-,right] (7) edge node[midway, below] {\tiny{$2$}} (9);
\path[draw,-,left] (8) edge node[midway, above] {\tiny{$2$}} (10);
\path[draw,-,right] (9) edge node[midway, above] {\tiny{$3$}} (10);

\end{tikzpicture}
}
\end{center}

\caption{The lattice of perfect matchings of a snake graph}
\label{Figure:PMLattice}
\end{figure}

\end{ex}

The twist admits a nice combinatorial interpretation. For every perfect matching $P$ of a snake graph $\SG$, we consider the symmetric difference $\operatorname{Sym}(\PM)=(\PM\cup\PM_{min})\backslash(\PM\cap\PM_{min})$. The set $\operatorname{Sym}(\PM)$ defines a union of cycles in the planar realization of $\SG$. We say that a tile $G$ of $\SG$ is \emph{twisted} if it lies in the interior of one of these cycles. The set of twisted tiles of $\PM$ is denoted by $\operatorname{Twist}(\PM)$. Furthermore, for an arc $\tau$ we define $\operatorname{Twist}(\PM)_{\tau}$ to be the subset of $\operatorname{Twist}(\PM)$ consisting of the tiles $G$ with face weight $\tau$. The name is justified by a theorem of Musiker--Schiffler--Williams according to which the set $\operatorname{Twist}(\PM)$ is equal to the set of tiles twisted along any shortest path from $\PM_{min}$ to $\PM$ in $L(\SG)$.

Musiker--Schiffler--Williams \cite[Theorem 5.2]{MSW2} show that the graph $L(\SG)$ is the Hasse diagram of the poset $\Match(\SG,\leq)$ when we define $\PM_1\leq\PM_2$ if and only if $\operatorname{Twist}(\PM_1)\subseteq\operatorname{Twist}(\PM_2)$. In fact, it is a distributive lattice with minimal element $\PM_{min}$ and maximal element $\PM_{max}$. Their theorem is a consequence of work by Propp \cite[Theorem 2]{Pr}.

\subsubsection{Laurent polynomials associated to snake graphs} \label{Sec:LPs-SGs}
Let $\cals$ be a surface with triangulation $\calt=\{\tau_1,\tau_2,\dots,\tau_n\}$. Let $\gamma$ be an arc in $\cals$ and $\SG$ its associated snake graph. Let $P$ be a perfect matching of $\SG.$ Then
\begin{itemize}
\item we assign formal variables $x_{\tau_i},y_{\tau_i}$ to each arc $\tau_i$ in $\calt$, and often use the abbreviations $x_i=x_{\tau_i}$ and $y_i=y_{\tau_i}$;
\item the \emph{weight monomial} $x(P)$ of $P$ is given by $x(P)= \displaystyle{\prod_{e\in P}} x_{w(e)}$;
\item the \emph{height monomial} $y(P)$ of $P$ is given by $y(P)= \displaystyle{\prod_{G \in \Tw(P)}} y_{w(G)}$; 
\item the \emph{crossing monomial} of $\gamma$ with respect to $\calt$ is given by $\cross (\gamma,\calt) = \displaystyle{\prod_{j\in J}} x_j$ where $J$ is the index set of the arcs in $\cals$ that $\gamma$ crosses;
\item the \emph{Laurent polynomial} associated to $\SG$ with respect to $\calt$ is defined as 
\[x_{\SG}= \frac{1}{\cross(\gamma,\calt)} \displaystyle{\sum_{P\models \SG}} x(P)y(P).\]
\end{itemize}

\begin{theorem}[\cite{MSW}]
\label{Thm:MSWExpansion}
Let $\cals$ be a surface with triangulation $\calt$ and let $\cala(\cals)$ be the cluster algebra associated to $\cals$. Let $\gamma$ be a (generalized) arc of $\cals$ and $\SG$ be its associated snake graph with respect to $\calt$ and $x_{\SG}$ be the Laurent polynomial associated to $\gamma.$ Then $x_{\SG}=x_{\gamma}.$ 
\end{theorem}

\begin{rem}Let $x_{\gamma}\in\A(\cals,\calt)$ be the cluster variable associated to an arc $\gamma$. The weight monomial from Section~\ref{Sec:LPs-SGs} is related to the $g$-vector of $x_{\gamma}$ with respect to the initial seed associated with $\T$, see \cite[Proposition 6.2]{MSW2}; namely
\begin{align}
\label{Eqn:gvector}
g(x_{\gamma})=\operatorname{deg}\left(x(\PM_{min})\right)-\operatorname{deg}\left(\operatorname{cross(\T,\gamma)}\right)\in\mathbb{Z}^n.
\end{align}
\end{rem}

In the above situation, for every perfect matching $P$ of $\SG_{\gamma}$ we put
\begin{align*}
\nu(P)=\left(\operatorname{deg}\left(x(\PM)\right)-\operatorname{deg}\left(\operatorname{cross(\T,\gamma)}\right),y(P)\right)\in\mathbb{Z}^{2n}.
\end{align*}

\subsection{Quantum cluster algebras} \label{Sec:QCA}

In this section we give a brief introduction to quantum cluster algebras. The section follows Berenstein--Zelevinsky \cite{BZ}. Let $q^{1/2}$ be an indeterminate and let $q^{-1/2}$ be its formal inverse.

\begin{defn}[Principal quantization]
Let $\widetilde{B}$ be an $n\times n$ skew-symmetric matrix. We consider the extension with principal coefficients $B$ of size $(2n)\times n$ and a $(2n)\times(2n)$ matrix $\Lambda$ as follows:
\begin{align*}
B=\left(\begin{matrix}\widetilde{B}\\I\end{matrix}\right),&&\Lambda=\left(\begin{matrix}0&-I\\I&-\widetilde{B}\end{matrix}\right).
\end{align*}  
We call the pair $(B,\Lambda)$ the \emph{principal quantization pair} of $\widetilde{B}$ and call $\Lambda$ the \emph{principal quasi-commutation matrix} of $\widetilde{B}$.
\end{defn}

Zelevinsky \cite[Example 0.5]{Z2} proves that the principal quantization pair $(B,\Lambda)$ attached to a skew-symmetric $n\times n$ matrix $\widetilde{B}$ is always \emph{compatible}, that is, the matrices obey the relation $B^T\Lambda=\left(\begin{matrix}I&0\end{matrix}\right)$.

Fomin--Zelevinsky's \emph{sign coherence conjecture} asserts that the entries in a $c$-vector are either all non-negative or all non-positive. Fomin--Zelevinsky \cite[Proposition 5.6]{FZ4} prove that the sign coherence conjecture is equivalent to the \emph{constant term conjecture} for $F$-polynomials. This conjecture holds true by a result of Derksen--Weyman--Zelevinsky \cite[Theorem 1.7]{DWZ}. Later different proofs of sign-coherence were given by Nagao \cite[Theorem 9.9]{Na}, Plamondon \cite[Theorem 3.13]{Pl} and Gross--Hacking--Keel--Kontsevich \cite[Corollary 5.5]{GHKK}. We use sign-coherence in the proof of the following proposition.

\begin{prop} 
\label{Prop:GMatrix}
Let $\widetilde{B}$ be an $n\times n$ skew-symmetric matrix. If we mutate the principal quantization pair $(B,\Lambda)$ along a sequence $\mathbf{i}=(i_r,\ldots,i_2,i_1)$ of mutable indices, then
\begin{align*}
\mu_{\mathbf{i}}(B,\Lambda)=(B_{\mathbf{i}},\Lambda_{\mathbf{i}})=\left(\left(\begin{matrix}\widetilde{B}_{\mathbf{i}}\\C_{\mathbf{i}}\end{matrix}\right),\left(\begin{matrix}0&-G_{\mathbf{i}}^T\\G_{\mathbf{i}}&G_{\mathbf{i}}\widetilde{B}_{\mathbf{i}}^{T}G_{\mathbf{i}}^T
\end{matrix}\right)\right).
\end{align*}
\end{prop}
\begin{proof}
By construction $B_{\mathbf{i}}=\mu_{\mathbf{i}}(B)$ is obtained from $B$ by ordinary matrix multiplication, hence the form of $B_{\mathbf{i}}$ follows from Definition \ref{Defn:Cmatrix}. We can show by induction that the upper left part of $\Lambda_{\mathbf{i}}$ is zero. Let $\mathbf{i}'=(i_{r_-1},\ldots,i_1)$. By definition $\Lambda_{\mathbf{i}}=E_{\epsilon}^T\Lambda_{\mathbf{i}'}E_{\epsilon}$, where we use the same notation as Berenstein--Zelevinsky \cite[Equation (3.2)]{BZ}. The sign coherence of the $c$ vectors implies there is a always a sign $\epsilon$ such that $E_{\epsilon}$ is block diagonal, which proves that the upper left part of $\Lambda_{\mathbf{i}}$ is zero. Moreover, the pair $(B_{\mathbf{i}},\Lambda_{\mathbf{i}})$ must satisfy the compatibility condition $B_{\mathbf{i}}^T\Lambda_{\mathbf{i}}=\left(\begin{matrix}I&0\end{matrix}\right)$ so that the lower left part of $\Lambda_{\mathbf{i}}$ is equal to $G_{\mathbf{i}}$ due to the relation $C_{\mathbf{i}}G_{\mathbf{i}}^T=I$. The upper right part of $\Lambda_{\mathbf{i}}$ must be equal to $-G_{\mathbf{i}}^T$ due to skew-symmetry. The structure of the lower right part can be read off from the compatibility condition.
\end{proof}

Suppose that $\Lambda$ is a skew-symmetric integer $m\times m$ matrix. The \emph{based quantum torus} $\QTorus(\Lambda)$ is a $\mathbb{Z}[q^{\pm 1/2}]$-algebra with $\mathbb{Z}[q^{\pm 1/2}]$-basis $\{\, M[a]\mid \mathbb{Z}^m\,\}$ indexed by $\mathbb{Z}^m$; on basis elements the multiplication is defined as
\begin{align}
\label{Eqn:MultiplicationBZBasis}
M[a]M[b]=q^{\frac{1}{2} a^T\Lambda b}M[a+b]
\end{align}
for all $a,b\in\mathbb{Z}^m$.

The based quantum torus $\QTorus(\Lambda)$ is an Ore domain and it is contained in its skew field of fractions $\mathcal{F}$. We refer the reader to \cite{BZ} for details about this construction. Note that $\mathcal{F}$ is an algebra over $\mathbb{Q}(q^{1/2})$. We say that two elements $f_1,f_2\in\mathcal{F}$ are $q$-\emph{commuting} if there exists an integer $k$ such that $f_if_j=q^{k/2}f_jf_i$. Examples of $q$-commuting elements include the basis elements of the based quantum torus: for all $a,b\in\mathbb{Z}^m$ we have $M[a]M[b]=q^{a^T\Lambda b}M[b]M[a]$.

A \emph{quantum cluster} is a tuple $(\x',\y')=(x_1'\ldots,x'_n,y'_1,\ldots,y'_{m-n})$ of $m$ pairwise $q$-commuting elements in $\mathcal{F}$. The elements $x'_i$ with $i\in [n]$ are called \emph{quantum cluster variables} and the elements $y'_i$ with $i\in [1,m-n]$ are called \emph{frozen variables}. The quantum cluster $(\x,\y)=\left(M[e_1],\ldots,M[e_m]\right)$, where $e_i$ denotes the $i$-th standard basis vector of $\mathbb{Z}^m$, is called the \emph{initial quantum cluster}. The exponents arising in the $q$-commutativity relations among the quantum cluster and frozen variables in a quantum cluster $(\x',\y')$ assemble a skew-symmetric matrix $\Lambda'$. We call $\Lambda'$ the $q$-\emph{commutativity matrix} of the quantum cluster. We denote the basis elements in the based quantum torus $\QTorus(\Lambda')$ by $M_{\Lambda'}[a]$ with $a\in\mathbb{Z}^m$. 

A \emph{quantum seed} is a tuple $(B',\Lambda',\x',\y')$ such that $(\x',\y')$ is a quantum cluster with $q$-commutativity matrix $\Lambda'$ and there exists a positive integer $k$ such that $(B')^T\Lambda=\left(\begin{matrix}kI&0\end{matrix}\right)$. The last property is called the \emph{compatibility condition}.

The \emph{mutation} $\mu_k(B',\Lambda',\x',\y')= (B'',\Lambda'',\x'',\y'')$ of the quantum seed  $(B',\Lambda',\x',\y')$ at $k\in [n]$ is constructed as follows. The quantum cluster $(\x'',\y'')$ is obtained from $(\x',\y')$ by replacing $x_k'$ with 
\begin{align*}
x_k'' = M_{\Lambda'}\left(-e_k + \sum_{b_{ki}>0} b_{ik}e_i\right) + M_{\Lambda'}\left(-e_k - \sum_{b_{ik}<0} b_{ik}e_i\right).
\end{align*}  
This cluster consists again of pairwise $q$-commuting elements by the compatibility condition. Then $\Lambda''$ is defined as the $q$-commutativity matrix of $(\x'',\y'')$. Lastly, $B''=\mu_k(B')$ is given by the usual mutation of exchange matrices.

The \emph{quantum cluster algebra} $\A_q(B,\Lambda,\x,\y)$ is the subalgebra of $\mathcal{F}$ generated by all quantum cluster variables and frozen variables in all quantum seeds that can be obtained from $(B,\Lambda,\x,\y)$ by sequences of mutations. We also use the shorthand notation $\A_q(B,\Lambda)$ for $\A_q(B,\Lambda,\x,\y)$.

\section{An expansion formula for quantum cluster algebras of type A}
\label{Sec:TypeA}

\subsection{A quantum expansion formula for type A }

Let $n\geq 1$ be a natural number. We consider the marked oriented surface $(\cals,\M)$ which is formed by a disc $\cals$ with a set $\M$ of $n+3$ marked points on the boundary. The cluster algebra $\A(\Sur,\M)$ has finite type $A_n$. We fix a triangulation $\T$ of $(\cals,\M)$, and we denote by $Q$ the quiver of $\T$. Without loss of generality we may assume that $Q$ is a quiver of type $A_n$.

Let $(B,\Lambda)$ be the principal quantization pair of the signed adjacency matrix of $Q$. We are interested in the quantum cluster algebra $\A_q(B,\Lambda)$. We denote the initial quantum cluster by $(\x,\y)=(x_i,y_i)_{i\in [1,n]}$.

Our goal is to prove the following theorem about the expansion of quantum cluster variables as $\mathbb{Z}[q^{\pm 1/2}]$-linear combinations of quantum Laurent monomials in the initial quantum cluster variables. As it turns out, the coefficients are all equal to $1$.

\begin{theorem}
\label{Thm:ExpansionA}
If $\gamma$ is any arc in $(\Sur,\M)$, then
\begin{align*}
x_{\gamma}=\sum_{\PM\models\SG_{\gamma}} M\left[\nu(\PM)\right]\in \A_q(B,\Lambda).
\end{align*}
\end{theorem}

The proof of Theorem \ref{Thm:ExpansionA} will be given at the end of Section \ref{Sec:TypeA}.

\subsection{Quantum F-polynomials}

The main tool in the proof of the expansion formula is quantum $F$-polynomials. These polynomials are a quantized version of the $F$-polynomials from Fomin--Zelevinsky's fourth cluster algebra article \cite{FZ4}. Quantum $F$-polynomials have been introduced and studied by Tran \cite{T}. Let us recall some notions and results from Tran's article. 

\begin{defn}[$\widehat{y}$-variables]
Suppose $i\in [1,n]$. Let $b_i$ be the $i$-th column of the exchange matrix $B$. We put $\widehat{y_i}=M[b_i]$. Moreover, for every subset $S=\{s_1,\ldots,s_r\}\subseteq [1,n]$ with $s_1<\ldots<s_r$ we set
\begin{align*}
\widehat{y}_S=q^{\frac{1}{2}\sum\limits_{1\leq i<j\leq r} b_{s_j,s_i}} \widehat{y}_{s_1}\cdots \widehat{y}_{s_r}.
\end{align*}
\end{defn}

For two indices $i,j\in [1,n]$ we have $b_i^T\Lambda b_j =-b_{ji}$ so that
\begin{align*}
\widehat{y}_i\widehat{y}_j=M[b_i]M[b_j]=q^{\tfrac12 b_i^T\Lambda b_j}M[b_i+b_j]=q^{-\tfrac12 b_{ji}}M[b_i+b_j]=q^{-b_{ji}}\widehat{y}_j\widehat{y}_i.
\end{align*}
We conclude that $\widehat{y}_S$ is a standard basis element in Berenstein--Zelevinsky's quantum torus. More precisely, if we denote by $b_S=\sum_{s\in S} b_s$ the sum of the column vectors of $B$ indexed by elements of $S$, then $\widehat{y}_S=M[b_S]$. 

We consider an arc $\gamma$ and the associated snake graph $\SG_{\gamma}$. The weights of the tiles of $\SG_{\gamma}$ are given by a subset $T\subseteq Q_0$ such that the full subquiver of $Q$ on the vertex set $T$ is connected.

\begin{defn}[Properties of subsets] Suppose that $S\subseteq T$.
\begin{enumerate}
\item We say that $S$ is \emph{successor-closed} if the following condition holds. If $j\in S$, $i\in T$ and there exists an arrow $j\to i$ in $Q$, then $i\in S$. 
\item We denote by $\Phi(S)$ the number of connected components of the full subquiver of $Q$ on the vertex set $S$. 
\end{enumerate}
\end{defn}

Note that successor-closed subsets of $T$ are in natural bijection with perfect matchings of $\SG_{\gamma}$. More precisely, for every successor-closed subset $S\subseteq T$ there is a unique perfect matching $P$ such that the weights of the twisted tiles of $P$ are equal to $S$.

\begin{ex}
Suppose that $T=\{1,2,3,4,5\}$ and that the full subquiver of $Q$ on the vertex set $T$ is given by the quiver on the left of Figure \ref{Figure:SuccessorClosed}. Then the set $S=\{1,3,4\}$ is successor-closed. In the middle of the figure we display the snake graph and its minimal perfect matching. Twisting the tiles labelled $1,3,4$ yields the perfect matching on the right of the figure.

In this example we have $\lvert S\rvert=3$ and $\Phi(S)=2$. Note that the number of arrows inside $S$ is equal to $1$, which is equal to the difference of the previous two numbers.
\end{ex}

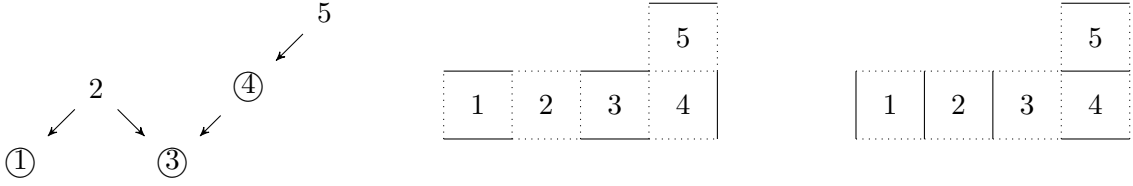
\begin{figure}
\begin{multicols}{3}

\begin{center}
\begin{tikzpicture}
\newcommand{\horizontalsize}{1cm}
\newcommand{\verticalsize}{1cm}

\node (1) at (0,0) {\textcircled{1}};
\node (2) at (\horizontalsize,\verticalsize) {2};
\node (3) at (2*\horizontalsize,0) {\textcircled{3}};
\node (4) at (3*\horizontalsize,\verticalsize) {\textcircled{4}};
\node (5) at (4*\horizontalsize,2*\verticalsize) {5};

\path[draw,->,shorten <=2pt,shorten >=2pt,>=stealth'] (2) to (1);
\path[draw,->,shorten <=2pt,shorten >=2pt,>=stealth'] (2) to (3);
\path[draw,->,shorten <=2pt,shorten >=2pt,>=stealth'] (4) to (3);
\path[draw,->,shorten <=2pt,shorten >=2pt,>=stealth'] (5) to (4);

\end{tikzpicture}
\end{center}

\columnbreak

\begin{center}
\begin{tikzpicture}

\newcommand{\dist}{0.9cm}

\path[draw] (1*\dist,0) edge node {} (2*\dist,0);
\path[draw,dotted] (2*\dist,0) edge node {} (3*\dist,0);
\path[draw] (3*\dist,0) edge node {} (4*\dist,0);
\path[draw,dotted] (4*\dist,0) edge node {} (5*\dist,0);

\path[draw] (1*\dist,\dist) edge node {} (2*\dist,\dist);
\path[draw,dotted] (2*\dist,\dist) edge node {} (3*\dist,\dist);
\path[draw] (3*\dist,\dist) edge node {} (4*\dist,\dist);
\path[draw,dotted] (4*\dist,\dist) edge node {} (5*\dist,\dist);

\path[draw] (4*\dist,2*\dist) edge node {} (5*\dist,2*\dist);

\path[draw,dotted] (1*\dist,0) edge node {} (1*\dist,\dist);
\path[draw,dotted] (2*\dist,0) edge node {} (2*\dist,\dist);
\path[draw,dotted] (3*\dist,0) edge node {} (3*\dist,\dist);
\path[draw,dotted] (4*\dist,0) edge node {} (4*\dist,\dist);
\path[draw] (5*\dist,0) edge node {} (5*\dist,\dist);

\path[draw,dotted] (4*\dist,\dist) edge node {} (4*\dist,2*\dist);
\path[draw,dotted] (5*\dist,\dist) edge node {} (5*\dist,2*\dist);

\node (T1) at (1.5*\dist,0.5*\dist) {$1$};
\node (T2) at (2.5*\dist,0.5*\dist) {$2$};
\node (T3) at (3.5*\dist,0.5*\dist) {$3$};
\node (T4) at (4.5*\dist,0.5*\dist) {$4$};
\node (T5) at (4.5*\dist,1.5*\dist) {$5$};

\end{tikzpicture}
\end{center}

\columnbreak

\begin{center}
\begin{tikzpicture}

\newcommand{\dist}{0.9cm}

\path[draw,dotted] (1*\dist,0) edge node {} (2*\dist,0);
\path[draw,dotted] (2*\dist,0) edge node {} (3*\dist,0);
\path[draw,dotted] (3*\dist,0) edge node {} (4*\dist,0);
\path[draw] (4*\dist,0) edge node {} (5*\dist,0);

\path[draw,dotted] (1*\dist,\dist) edge node {} (2*\dist,\dist);
\path[draw,dotted] (2*\dist,\dist) edge node {} (3*\dist,\dist);
\path[draw,dotted] (3*\dist,\dist) edge node {} (4*\dist,\dist);
\path[draw] (4*\dist,\dist) edge node {} (5*\dist,\dist);

\path[draw] (4*\dist,2*\dist) edge node {} (5*\dist,2*\dist);

\path[draw] (1*\dist,0) edge node {} (1*\dist,\dist);
\path[draw] (2*\dist,0) edge node {} (2*\dist,\dist);
\path[draw] (3*\dist,0) edge node {} (3*\dist,\dist);
\path[draw,dotted] (4*\dist,0) edge node {} (4*\dist,\dist);
\path[draw,dotted] (5*\dist,0) edge node {} (5*\dist,\dist);

\path[draw,dotted] (4*\dist,\dist) edge node {} (4*\dist,2*\dist);
\path[draw,dotted] (5*\dist,\dist) edge node {} (5*\dist,2*\dist);

\node (T1) at (1.5*\dist,0.5*\dist) {$1$};
\node (T2) at (2.5*\dist,0.5*\dist) {$2$};
\node (T3) at (3.5*\dist,0.5*\dist) {$3$};
\node (T4) at (4.5*\dist,0.5*\dist) {$4$};
\node (T5) at (4.5*\dist,1.5*\dist) {$5$};

\end{tikzpicture}
\end{center}

\end{multicols}

\caption{A successor-closed subset and the corresponding perfect matching}
\label{Figure:SuccessorClosed}

\end{figure}

We can generalize the observation in the following way. 

\begin{prop}
\label{Prop:Successor}
Suppose that $S\subseteq T$ is successor-closed. Then
\begin{align*}
\lvert S\rvert-\vert \left\lbrace\, (k,j)\in S\times T \mid (k,j)\in Q_1\,\right\rbrace\rvert = \Phi(S).
\end{align*}
\end{prop}

\begin{proof}
The full subquiver of $Q$ on the vertex set $S$ is a (directed) forest, i.e. a disjoint union of (directed) trees. In every connected component the difference between the number of vertices and the number of arrows is equal to $1$.  
\end{proof}

Recall that $g_{\gamma}$ denotes the $g$-vector of the cluster variable $x_{\gamma}$. Tran, see \cite[Proposition 7.3]{T}, shows that the entry of the $g$-vector at an index $s\in Q_0$ is equal to
\begin{align}
\label{Eqn:Trangvector}
(g_{\gamma})_s=\lvert\left\lbrace\,  j\in T\mid \exists (s\to j)\in Q_1 \,\right\rbrace\rvert-1.
\end{align}
 In the next statement we view the $g$-vector as an element in $\mathbb{Z}^{2n}$ via the canonical inclusion $\mathbb{Z}^{n} \hookrightarrow \mathbb{Z}^{2n}$.

\begin{theorem}[Tran \cite{T}, Theorems 5.3, 7.4]
\label{Thm:Tran}
If we put
\begin{align*}
F_{\gamma}=\sum_{S\subseteq T} q^{\frac{1}{2} \Phi(S)} \widehat{y}_S
\end{align*}
where the sum runs over all successor-closed subsets $S$ of $T$, then
\begin{align*}
x_{\gamma}=F_{\gamma}M[g_{\gamma}].
\end{align*}
\end{theorem}

The element $F_{\gamma}$ is called the \emph{quantum $F$-polynomial}.

\begin{proof}[Proof of Theorem \ref{Thm:ExpansionA}] We apply Tran's Theorem \ref{Thm:Tran}. Using $\widehat{y}_S=M[b_S]$, where $b_S$ is the sum of the column vectors of $B$ indexed by $S$, we obtain
\begin{align*}
x_{\gamma}=\sum_{S} q^{\tfrac12 \Phi(S)} M[b_S]M[g_{\gamma}]=\sum_{S} q^{\tfrac12 \Phi(S)+\tfrac12 b_S^T \Lambda g_{\gamma}} M[b_s+g_{\gamma}]. 
\end{align*}
We show that $ \Phi(S)+ b_S^T \Lambda g_{\gamma}$ is zero for every successor-closed subset $S\subseteq T$. This claim completes the proof because it implies that every coefficient in the expansion of $x_{\gamma}$ as a linear combination of standard basis elements is equal to $1$. Recall that the last $n$ entries of $g_{\gamma}$ are $0$ by convention. This implies $b_i^T\Lambda g_{\gamma}=(g_{\gamma})_i$ for every $i\in S$. We use equation (\ref{Eqn:Trangvector}), sum the expression over all $i\in S$ and apply Proposition \ref{Prop:Successor} to obtain the result.
\end{proof}

\section{An expansion formula for the Kronecker quantum cluster algebra}
\label{Section:Kronecker}

\subsection{The setup}

We consider the marked oriented surface $(\Sur,\M)$ which is formed by an annulus $\Sur$ and a set $\M\subseteq \partial \Sur$ of two marked points, one on each boundary component. Figure \ref{Figure:InitialSeed} shows a triangulation $\T$ of $(\Sur,\M)$ by two arcs $\tau_1$, $\tau_2$. The quiver of $\T$ is the Kronecker quiver $1\rightrightarrows2$. The quiver of its principal extension, denoted $Q$, is also shown in Figure \ref{Figure:InitialSeed}. Note that $Q$ is an ice quiver with two mutable vertices $1,2$ and two frozen vertices $1',2'$. 

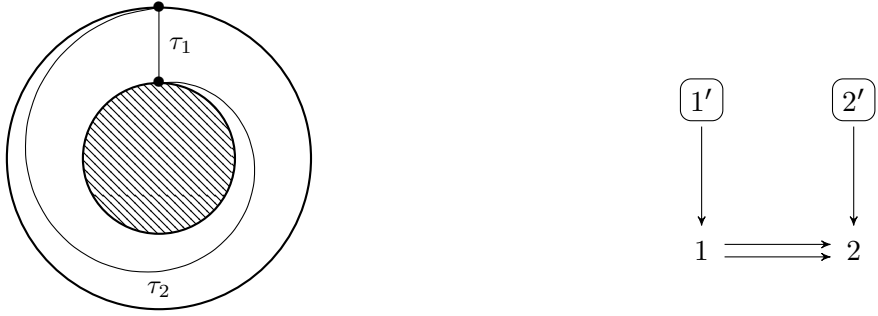
\begin{figure}

\begin{multicols}{2}

\begin{center}
\begin{tikzpicture}

\newcommand{\RadiusSmallCircle}{1cm}
\newcommand{\RadiusLargeCircle}{2cm}
\newcommand{\AngleInnerPoint}{90} 
\newcommand{\AngleOuterPoint}{90} 

\draw[pattern=north west lines,thick] (0,0) circle (\RadiusSmallCircle);
\draw (\AngleInnerPoint:\RadiusSmallCircle) node {$\bullet$}; 

\draw[thick] (0,0) circle (\RadiusLargeCircle);
\draw (\AngleOuterPoint:\RadiusLargeCircle) node {$\bullet$}; 

\draw[domain=0:1,smooth,variable=\t] plot (\AngleInnerPoint-\t*\AngleInnerPoint+\t*\AngleOuterPoint:\t*\RadiusLargeCircle-\t*\RadiusSmallCircle+\RadiusSmallCircle);
\node[right] (1) at (0.5*\AngleInnerPoint+0.5*\AngleOuterPoint:0.5*\RadiusLargeCircle+0.5*\RadiusSmallCircle) {\small{$\tau_1$}};

\draw[domain=0:1,smooth,variable=\t] plot (\AngleInnerPoint-\t*\AngleInnerPoint-\t*360+\t*\AngleOuterPoint:\t*\RadiusLargeCircle-\t*\RadiusSmallCircle+\RadiusSmallCircle);
\node[below] (2) at (0.5*\AngleInnerPoint-180+0.5*\AngleOuterPoint:0.5*\RadiusLargeCircle+0.5*\RadiusSmallCircle) {\small{$\tau_2$}};

\end{tikzpicture}
\end{center}

\columnbreak

\null \vfill
\begin{center}
\begin{tikzpicture}

\newcommand{\distance}{2}

\node (1) at (0,0) {$1$};
\node (2) at (\distance,0) {$2$};

\node[draw,rectangle,rounded corners] (3) at (0,\distance) {$1'$};
\node[draw,rectangle,rounded corners] (4) at (\distance,\distance) {$2'$};

\path[draw,->,shorten <=2pt,shorten >=2pt,>=stealth'] (3) to (1);
\path[draw,->,shorten <=2pt,shorten >=2pt,>=stealth'] (4) to (2);

\path[draw,->,shorten <=2pt,shorten >=2pt,>=stealth',transform canvas={yshift=0.5ex}] (1) to (2);
\path[draw,->,shorten <=2pt,shorten >=2pt,>=stealth',transform canvas={yshift=-0.5ex}] (1) to (2);

\end{tikzpicture}
\end{center}
\null \vfill

\end{multicols}

\caption{A triangulation of an annulus and its associated quiver}
\label{Figure:InitialSeed}
\end{figure}

The triangulation gives rise to a seed of the associated cluster algebra $\A(\x,\y,B)$. The set of cluster variables of this cluster algebra admits a natural parametrization by the set of integers $\mathbb{Z}$, see Figure \ref{Figure:Clusters}. Hence the set of cluster variables of $\A(\x,\y,B)$ can be written as $\{\,x_n\mid n\in\mathbb{Z}\,\}$; in this notation the exchange relations become $x_{n-1}x_{n+1}=x_n^2+y_1^{n-1}y_2^{n-2}$ for $n\geq 2$ and equations for $n\leq 1$ can be written down similarly.

\begin{figure}

\begin{center}
\begin{tikzpicture}

\newcommand{\hdist}{2cm} 
\newcommand{\vdist}{1cm} 

\node (-1) at (-2*\hdist,0) {$(x_{-1},x_0,y_1,y_2)$};
\node (0) at (-\hdist,\vdist) {$(x_0,x_1,y_1,y_2)$};
\node (1) at (0,0) {$(x_1,x_2,y_1,y_2)$};
\node (2) at (\hdist,\vdist) {$(x_2,x_3,y_1,y_2)$};
\node (3) at (2*\hdist,0) {$(x_3,x_4,y_1,y_2)$};
\node (4) at (3*\hdist,\vdist) {$(x_4,x_5,y_1,y_2)$};

\path[draw,-] (-1) to (0);
\path[draw,-] (0) to (1);
\path[draw,-] (1) to (2);
\path[draw,-] (2) to (3);
\path[draw,-] (3) to (4);

\node (LeftDots) at (-3*\hdist,0.5*\vdist) {$\hdots$};
\node (RightDots) at (4*\hdist,0.5*\vdist) {$\hdots$};

\end{tikzpicture}
\end{center}

\caption{The exchange graph of the Kronecker cluster algebra}
\label{Figure:Clusters}
\end{figure}
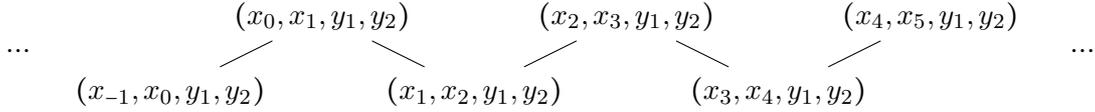

Fomin--Zelevinsky's Laurent phenomenon asserts that $x_n\in\mathbb{Z}[x_1^{\pm 1},x_2^{\pm 1},y_1,y_2]$ for every $n\in\mathbb{Z}$. There are two different explicit formulae to write $x_n$ as a Laurent polynomial in $\x$. The first formula we state is a variation of a theorem by Caldero--Zelevinsky \cite[Theorem 4.1]{CZ}. The authors establish their formula by computing Euler characteristics of quiver Grassmannians  and by using Caldero--Chapoton's formula \cite{CC}. Later Zelevinsky \cite{Z} gave a simple inductive proof of the formula.

\begin{theorem}[cf. Caldero--Zelevinsky]
\label{Thm:BinomialCoefficientFormula}
For every $n\geq 0$ we have
\begin{align*}
x_{n+3}=x_1^{-n-1}x_2^{n+2}+\sum_{p+r\leq n}\binom{n-r}{p}\binom{n+1-p}{r}x_1^{2p-n-1}x_2^{2r-n}y_1^{n+1-r}y_2^p.
\end{align*}
\end{theorem}

 It is easy to see that the cluster variable $x_{n+3}$ is attached to an arc $\gamma$ in $(\Sur,\M)$ that crosses $\tau_1$ exactly $n+1$ times and crosses $\tau_2$ exactly $n$ times. Moreover, it can be seen that the snake graph $\Gn$ associated to $\gamma$ is given as follows.

\begin{defn}[Snake graphs of Kronecker type] 
\label{Defn:KroneckerSnake}
The snake graph $\Gn$ is the following straight snake graph consisting of $2n+1$ tiles of alternating face weight $1$ and $2$ and with exactly $n+1$ tiles of weight $1$ and exactly $n$ tiles of weight $2$; the top and bottom edges of the tile of face weight $1$ have edge weights $2$ and the tile of face weight $2$ has edge weights $1$:
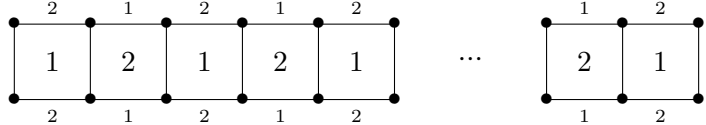
\begin{center}
\begin{tikzpicture}

\newcommand{\dist}{1cm} 

\node (-1down) at (-1*\dist,0) {$\bullet$};
\node (-2down) at (-2*\dist,0) {$\bullet$};
\node (-3down) at (-3*\dist,0) {$\bullet$};
\node (-4down) at (-4*\dist,0) {$\bullet$};
\node (-5down) at (-5*\dist,0) {$\bullet$};
\node (-6down) at (-6*\dist,0) {$\bullet$};

\node (-1up) at (-1*\dist,\dist) {$\bullet$};
\node (-2up) at (-2*\dist,\dist) {$\bullet$};
\node (-3up) at (-3*\dist,\dist) {$\bullet$};
\node (-4up) at (-4*\dist,\dist) {$\bullet$};
\node (-5up) at (-5*\dist,\dist) {$\bullet$};
\node (-6up) at (-6*\dist,\dist) {$\bullet$};

\path[draw,-,above] (-1*\dist,\dist) edge node {\tiny{$2$}} (-2*\dist,\dist);
\path[draw,-,above] (-2*\dist,\dist) edge node {\tiny{$1$}} (-3*\dist,\dist);
\path[draw,-,above] (-3*\dist,\dist) edge node {\tiny{$2$}} (-4*\dist,\dist);
\path[draw,-,above] (-4*\dist,\dist) edge node {\tiny{$1$}} (-5*\dist,\dist);
\path[draw,-,above] (-5*\dist,\dist) edge node {\tiny{$2$}} (-6*\dist,\dist);

\path[draw,-,below] (-1*\dist,0) edge node {\tiny{$2$}} (-2*\dist,0);
\path[draw,-,below] (-2*\dist,0) edge node {\tiny{$1$}} (-3*\dist,0);
\path[draw,-,below] (-3*\dist,0) edge node {\tiny{$2$}} (-4*\dist,0);
\path[draw,-,below] (-4*\dist,0) edge node {\tiny{$1$}} (-5*\dist,0);
\path[draw,-,below] (-5*\dist,0) edge node {\tiny{$2$}} (-6*\dist,0);

\path[draw,-] (-1*\dist,0) edge node {} (-1*\dist,\dist);
\path[draw,-] (-2*\dist,0) edge node {} (-2*\dist,\dist);
\path[draw,-] (-3*\dist,0) edge node {} (-3*\dist,\dist);
\path[draw,-] (-4*\dist,0) edge node {} (-4*\dist,\dist);
\path[draw,-] (-5*\dist,0) edge node {} (-5*\dist,\dist);
\path[draw,-] (-6*\dist,0) edge node {} (-6*\dist,\dist);

\node (1down) at (1*\dist,0) {$\bullet$};
\node (2down) at (2*\dist,0) {$\bullet$};
\node (3down) at (3*\dist,0) {$\bullet$};

\node (1up) at (1*\dist,\dist) {$\bullet$};
\node (2up) at (2*\dist,\dist) {$\bullet$};
\node (3up) at (3*\dist,\dist) {$\bullet$};

\path[draw,-,above] (1*\dist,\dist) edge node {\tiny{$1$}} (2*\dist,\dist);
\path[draw,-,above] (2*\dist,\dist) edge node {\tiny{$2$}} (3*\dist,\dist);

\path[draw,-,below] (1*\dist,0) edge node {\tiny{$1$}} (2*\dist,0);
\path[draw,-,below] (2*\dist,0) edge node {\tiny{$2$}} (3*\dist,0);

\path[draw,-] (1*\dist,0) edge node {} (1*\dist,\dist);
\path[draw,-] (2*\dist,0) edge node {} (2*\dist,\dist);
\path[draw,-] (3*\dist,0) edge node {} (3*\dist,\dist);

\node (Dots) at (0,0.5*\dist) {$\hdots$};

\node (-T5) at (-5.5*\dist,0.5*\dist) {$1$};
\node (-T4) at (-4.5*\dist,0.5*\dist) {$2$};
\node (-T3) at (-3.5*\dist,0.5*\dist) {$1$};
\node (-T2) at (-2.5*\dist,0.5*\dist) {$2$};
\node (-T1) at (-1.5*\dist,0.5*\dist) {$1$};

\node (T2) at (2.5*\dist,0.5*\dist) {$1$};
\node (T1) at (1.5*\dist,0.5*\dist) {$2$};

\end{tikzpicture}
\end{center}
For some discussion which will become clear later, we also need the snake graph $\Hn$, which is obtained from $\Gn$ by removing the last tile (with weight $1$). Note that $\Hn$ contains exactly $n$ tiles with face weight $1$ and exactly $n$ tiles with face weight $2$. 
\end{defn}

The perfect matching $\PM_{min}$ of $\Gn$ is formed by all the edges of weight $1$ and its weight monomial is equal to $x(\PM_{min})=x_2^{2n+2}$. The following statement is a precedent of Theorem \ref{Thm:MSWExpansion} due to \cite{MP,MS}.

\begin{theorem}[Musiker--Propp, Musiker--Schiffler]
\label{Thm:SnakeGraphFormula}
Let $n\geq 0$. Then we have
\begin{align*}
x_{n+3}=\frac{1}{x_1^{n+1}x_2^n}\left(\sum_{\PM\models\Gn}x(\PM)y(\PM)\right).
\end{align*}
\end{theorem}

\begin{rem} Suppose that $n \geq 0$. The cluster variable $x_{-n}$ corresponds to an arc in $(\mathcal{S},M)$ that crosses $\tau_1$ exactly $n$ times and crosses $\tau_2$ exactly $n+1$ times. Its snake graph is obtained from $\mathcal{G}_n$ by reversing the roles of the indices $1$ and $2$. There are formulae for $x_{-n}$ analogous to Theorem~\ref{Thm:BinomialCoefficientFormula} and \ref{Thm:SnakeGraphFormula} and all other statements made later in the text. For better readibility we refrain from writing down statements for both $x_{-n}$ and $x_{n+3}$ and focus on (quantum) cluster variables with positive indices throughout the text.
\end{rem}

In the following we abbreviate $\operatorname{Twist}(\PM)_{\tau_i}$ by $\operatorname{Twist}(\PM)_i$ for $i\in \{1,2\}$.

\begin{defn}[Level sets] Let $p,r,n$ be natural numbers. The \emph{level set} $\Match(\Gn)_{p,r}$ is the set of perfect matchings $\PM$ of $\Gn$ with $\lvert\operatorname{Twist}(\PM)_1\rvert=n+1-r$ and $\lvert\operatorname{Twist}(\PM)_2\rvert=p$. (In other words, $\Match(\Gn)_{p,r}$ is the set of perfect matchings obtained from $\PM_{min}$ by twisting $n+1-r$ tiles of weight $1$ and $p$ tiles of weight $2$.) In this case we also write $\PM\models_p^r\Gn$. Similarly, we let $\Match(\Hn)_{p,r}$ be the set of perfect matchings $\PM$ of $\Hn$ with $\lvert\operatorname{Twist}(\PM)_1\rvert=n-r$ and $\lvert\operatorname{Twist}(\PM)_2\rvert=p$. In this case we also write $\PM\models_p^r\Hn$. 
\end{defn}

\begin{rem}
\label{Rem:NumberOf Matchings}
Let $p,r,n$ be natural numbers.
\begin{itemize}
\item[(a)] Let $\PM\in\Match(\Gn)\backslash\{\PM_{min}\}$. If a tile $G\in\operatorname{Twist}(\PM)$ has face weight $2$, then its neighboring tiles $G'$ and $G''$ must belong to $\operatorname{Twist}(\PM)$ as well. In particular, we must have $\lvert\operatorname{Twist}(\PM)_1\rvert\geq \lvert\operatorname{Twist}(\PM)_2\rvert+1$. It follows that $\Match(\Gn)_{p,r}=\varnothing$ unless $p+r\leq n$ or $(p,r)=(0,n+1)$.
\item[(b)] A perfect matching satisfies $\PM\models_p^r\Gn$ if and only if $y(\PM)=y_1^{n+1-r}y_2^p$.
\item[(c)] Combining Theorems \ref{Thm:BinomialCoefficientFormula} and \ref{Thm:SnakeGraphFormula} we see that $\lvert\Match(\Gn)_{p,r}\rvert=\binom{n-r}{p}\binom{n+1-p}{r}$ if $p+r\leq n$. Note that there is a direct combinatorial proof of this identity by Musiker--Propp \cite[Section 2.1]{MP}. The same authors also show that $\lvert\Match(\Hn)_{p,r}\rvert=\binom{n-r}{p}\binom{n-p}{r}$ if $p+r\leq n$.
\end{itemize}
\end{rem}

\subsection{The quantum cluster algebra of Kronecker type}

We consider the quantum cluster algebra $\A_q(B,\Lambda)$ constructed from the principal quantization pair
\begin{align*}
B=\left(\begin{matrix}
\widetilde{B}\\I
\end{matrix}\right)
=\left(\begin{matrix}
0&2\\
-2&0\\
1&0\\
0&1
\end{matrix}\right),&&
\Lambda=\left(\begin{matrix}
0&-I\\
I&-\widetilde{B}
\end{matrix}\right)
=\left(\begin{matrix}
0&0&-1&0\\
0&0&0&-1\\
1&0&0&-2\\
0&1&2&0
\end{matrix}\right).
\end{align*}
The initial seed of $\A_q(B,\Lambda)$ is denoted by $(\x,\y,B,\Lambda)$. In particular, we have $x_1x_2=x_2x_1$, $y_1y_2=q^{-2}y_2y_1$ and $x_iy_i=q^{-1}y_ix_i$ for $i\in\{1,2\}$.

By a slight abuse of notation we denote the quantum cluster variables of $\Aq(B,\Lambda)$ again by $x_n$ with $n\in\mathbb{Z}$. By equation (\ref{Eqn:gvector}) the $g$-vector of $x_{n+3}$ is
\begin{align}
\label{Eqn:gvectorexplicit}
g(x_{n+3})=\operatorname{deg}(x(\PM_{min}))-\operatorname{deg}(\operatorname{cross}(\T,\gamma))=\left(\begin{matrix}0\\2n+2\end{matrix}\right)-\left(\begin{matrix}n+1\\n\end{matrix}\right)=\left(\begin{matrix}-n-1\\n+2\end{matrix}\right)
\end{align}
where $\gamma$ is the arc corresponding to $x_{n+3}$. From this we can deduce that the $G$- and $C$-matrix for the cluster $(x_{n+3},x_{n+2},y_1,y_2)$ are
\begin{align*}
G=\left(\begin{matrix}-n-1&-n\\n+2&n+1\end{matrix}\right),&&C=G^{-T}=\left(\begin{matrix}-n-1&n+2\\-n&n+1\end{matrix}\right).
\end{align*}
So according to Proposition \ref{Prop:GMatrix} we obtain the $q$-commutativity relations $x_{n+3}y_1=q^{n+1}y_1x_{n+3}$, $x_{n+3}y_2=q^{-n-2}y_2x_{n+3}$, and $y_1y_2=q^{-2}y_2y_1$.
These relations imply that
\begin{align*}
x_{n+2}^{-1}y_1^{n+2}y_2^{n+1}=q^{1-2(n+2)(n+1)}y_2^{n+1}y_1^{n+2}x_{n+2}^{-1}. 
\end{align*}
Hence the exchange relations become 
\begin{align*}
x_{n+2}x_{n+4}=x_{n+3}^2+q^{-1/2+(n+2)(n+1)}y_1^{n+2}y_2^{n+1} \quad (n\geq 0),\\
x_{n+4}x_{n+2}=x_{n+3}^2+q^{1/2+(n+2)(n+1)}y_1^{n+2}y_2^{n+1} \quad (n\geq 0).
\end{align*}

\subsection{An expansion formula via quantum binomial coefficients}

Let us recall some definitions from quantum algebra. Given two natural numbers $k,n$ with $k\leq n$. The \emph{quantum integer}, the \emph{quantum factorial} and the \emph{quantum binomial coefficient} are defined as\begin{align*}
[n]_q=\frac{q^{n/2}-q^{-n/2}}{q^{1/2}-q^{-1/2}},&&[n]_q!=\prod_{k=1}^{n}[k]_q,&&\left[n \atop k\right]_q=\frac{[n]_q!}{[k]_q![n-k]_q!}&&\in\mathbb{Q}(q^{\pm 1/2}).
\end{align*}
For example, $[0]_q=0$, $[1]_q=1$ and $[2]_q=q^{1/2}+q^{-1/2}$. Using the geometric series we may write a quantum integer as $[n]_q=\sum_{k=0}^{n-1}q^{(2k+1-n)/2}\in\mathbb{Z}[q^{\pm 1/2}]$. More generally, it is well-known that $\left[n \atop k\right]_q$ also lies in the smaller ring $\mathbb{Z}[q^{\pm 1/2}]$.

For all $n,k\geq 0$ a quantum version of Pascal's rule asserts that 
\begin{align}
\label{Eqn:QuantumPascal}
\left[n \atop k\right]_q=q^{-\frac{n-k}{2}}\left[n-1 \atop k-1\right]_q+q^{\frac{k}{2}}\left[n-1 \atop k\right]_q.
\end{align}

Moreover, quantum integers, factorials and binomial coefficients specialize to the ordinary integers, factorials and binomial coefficients, respectively, in the limit $q=1$.

The following theorem is a generalization of Theorem \ref{Thm:BinomialCoefficientFormula} to the quantum cluster algebra $\Aq(B,\Lambda)$. For related formulae for the quantum cluster variables in quantum cluster algebras of Kronecker type with different quantizations see Rupel \cite[Proposition 1.1]{R} and Lampe \cite[Theorem 5.3]{L}. Sz\'{a}nt\'{o} \cite[Theorem 4.1]{S} provides another quantization of Caldero--Zelevinsky's theorem by counting cardinalities of Kronecker quiver Grassmannians over finite fields. 

\begin{theorem}
\label{Thm:QuantumBinomial}
For $n\geq 0$ the following relation holds:
\begin{align*}
x_{n+3}=x_1^{-n-1}x_2^{n+2}+\sum_{p+r\leq n}\left[n-r\atop p\right]_q\left[n+1-p\atop r\right]_q M\left[2p-n-1,2r-n,n+1-r,p\right].
\end{align*}
\end{theorem}

We will sketch a proof of the theorem later in this section. First, let us introduce some notation.

\begin{defn}[Quantum loop element]
\label{Def:s1}
Put 
\begin{align*}
s_1=\sum_{P\models\SH_1} M\left[v(P)\right].
\end{align*}
\end{defn}
\begin{rem} \begin{itemize}
\item[(a)] The snake graph $\SH_1$ has exactly $3$ perfect matchings, see Figure \ref{Figure:H1}. Hence we can write explicitly $s_1=M[-1,1,0,0]+M[-1,-1,1,0]+M[1,-1,1,1]$.
\item[(b)] The element belongs to the quantum cluster algebra. More precisely,
\begin{align*}
s_1&=M[1,-1,1,1]+M[-1,-1,1,0]+M[-1,1,0,0]+qM[1,1,0,1]-qM[1,1,0,1]\\
&=\left(M[2,-1,0,1]+M[0,-1,0,0]\right)\left(M[-1,0,1,0]+M[-1,2,0,0]\right)-qM[1,1,0,1]\\
&=x_0x_3-q^{1/2}x_1x_2y_2\in \Aq(B,\Lambda).
\end{align*} 
\item[(c)] The name $s_1$ is chosen in accordance with the name of the corresponding non-quantized element in the classical cluster algebra $\A(\x,\y,B)$, see Zelevinsky \cite{Z}. This element plays a crucial role in the construction of bases of $\A(\x,\y,B)$. It is associated with the loop inside the annulus.
\end{itemize}
\end{rem}

\begin{figure}
\begin{center}
\begin{tikzpicture}

\newcommand{\dist}{1cm} 

\node (1down) at (1*\dist,0) {$\bullet$};
\node (2down) at (2*\dist,0) {$\bullet$};
\node (3down) at (3*\dist,0) {$\bullet$};

\node (1up) at (1*\dist,\dist) {$\bullet$};
\node (2up) at (2*\dist,\dist) {$\bullet$};
\node (3up) at (3*\dist,\dist) {$\bullet$};

\path[draw,dashed,above] (1*\dist,\dist) edge node {\tiny{$1$}} (2*\dist,\dist);
\path[draw,thick,above] (2*\dist,\dist) edge node {\tiny{$2$}} (3*\dist,\dist);

\path[draw,dashed,below] (1*\dist,0) edge node {\tiny{$1$}} (2*\dist,0);
\path[draw,thick,below] (2*\dist,0) edge node {\tiny{$2$}} (3*\dist,0);

\path[draw,-,thick] (1*\dist,0) edge node {} (1*\dist,\dist);
\path[draw,dashed] (2*\dist,0) edge node {} (2*\dist,\dist);
\path[draw,dashed] (3*\dist,0) edge node {} (3*\dist,\dist);

\node (Dots) at (0,0.5*\dist) {$P_{min}$};

\node (T2) at (2.5*\dist,0.5*\dist) {$1$};
\node (T1) at (1.5*\dist,0.5*\dist) {$2$};

\end{tikzpicture}\hspace{1cm}
\begin{tikzpicture}

\newcommand{\dist}{1cm} 

\node (1down) at (1*\dist,0) {$\bullet$};
\node (2down) at (2*\dist,0) {$\bullet$};
\node (3down) at (3*\dist,0) {$\bullet$};

\node (1up) at (1*\dist,\dist) {$\bullet$};
\node (2up) at (2*\dist,\dist) {$\bullet$};
\node (3up) at (3*\dist,\dist) {$\bullet$};

\path[draw,dashed,above] (1*\dist,\dist) edge node {\tiny{$1$}} (2*\dist,\dist);
\path[draw,dashed,above] (2*\dist,\dist) edge node {\tiny{$2$}} (3*\dist,\dist);

\path[draw,dashed,below] (1*\dist,0) edge node {\tiny{$1$}} (2*\dist,0);
\path[draw,dashed,below] (2*\dist,0) edge node {\tiny{$2$}} (3*\dist,0);

\path[draw,-,thick] (1*\dist,0) edge node {} (1*\dist,\dist);
\path[draw,-,thick] (2*\dist,0) edge node {} (2*\dist,\dist);
\path[draw,-,thick] (3*\dist,0) edge node {} (3*\dist,\dist);

\node (Dots) at (0,0.5*\dist) {$P_{med}$};

\node (T2) at (2.5*\dist,0.5*\dist) {$1$};
\node (T1) at (1.5*\dist,0.5*\dist) {$2$};

\end{tikzpicture}\hspace{1cm}
\begin{tikzpicture}

\newcommand{\dist}{1cm} 

\node (1down) at (1*\dist,0) {$\bullet$};
\node (2down) at (2*\dist,0) {$\bullet$};
\node (3down) at (3*\dist,0) {$\bullet$};

\node (1up) at (1*\dist,\dist) {$\bullet$};
\node (2up) at (2*\dist,\dist) {$\bullet$};
\node (3up) at (3*\dist,\dist) {$\bullet$};

\path[draw,-,thick,above] (1*\dist,\dist) edge node {\tiny{$1$}} (2*\dist,\dist);
\path[draw,dashed,above] (2*\dist,\dist) edge node {\tiny{$2$}} (3*\dist,\dist);

\path[draw,-,thick,below] (1*\dist,0) edge node {\tiny{$1$}} (2*\dist,0);
\path[draw,dashed,below] (2*\dist,0) edge node {\tiny{$2$}} (3*\dist,0);

\path[draw,dashed] (1*\dist,0) edge node {} (1*\dist,\dist);
\path[draw,dashed] (2*\dist,0) edge node {} (2*\dist,\dist);
\path[draw,-,thick] (3*\dist,0) edge node {} (3*\dist,\dist);

\node (Dots) at (0,0.5*\dist) {$P_{max}$};

\node (T2) at (2.5*\dist,0.5*\dist) {$1$};
\node (T1) at (1.5*\dist,0.5*\dist) {$2$};

\end{tikzpicture}
\end{center}
\caption{The perfect matchings of $\SH_1$}
\label{Figure:H1}
\end{figure}

\begin{lemma}
\label{Lemma:Recursion}
The equation $x_ns_1=x_{n+1}+q x_{n-1}y_1y_2$ holds for every $n\geq 2$.
\end{lemma}

\begin{proof}
We prove the lemma by induction on $n$. Suppose that $n=2$. By definition we have 
\begin{align*}
x_2s_1&=x_2\left(M[-1,1,0,0]+M[-1,-1,1,0]+M[1,-1,1,1]\right)\\
&=x_1^{-1}x_2^2+q^{-1/2}x_1^{-1}y_1+q x_1y_1y_2=x_3+q x_1y_1y_2. 
\end{align*}
For the induction step it is enough to show that for every $n\geq 3$,
\begin{align*}
x_n^{-1}x_{n+1}+q x_n^{-1}x_{n-1}y_1y_2&=x_{n-1}^{-1}x_n+q x_{n-1}^{-1}x_{n-2}y_1y_2\\
\Leftrightarrow x_{n-1}x_{n+1}+q x_{n-1}^2y_1y_2&=x_n^2+q x_{n}x_{n-2}y_1y_2.
\end{align*}
By the exchange relations the last equation is equivalent to
\begin{align*}
x_n^2+q^{-1/2+(n-1)(n-2)} y_1^{n-1}y_2^{n-2}+q x_{n-1}^2y_1y_2&=x_n^2+q \left(x_{n-1}^2+q^{1/2+(n-2)(n-3)} y_1^{n-2}y_2^{n-3}\right) y_1y_2\\
\Leftrightarrow q^{-1/2+(n-1)(n-2)} y_1^{n-1}y_2^{n-2}&=q q^{1/2+(n-2)(n-3)} q^{2(n-3)} y_1^{n-1}y_2^{n-2}.
\end{align*}
This equation is true so that we have proved the lemma.
\end{proof}

\begin{proof}[Sketch of the proof of Theorem \ref{Thm:QuantumBinomial}] Using Lemma \ref{Lemma:Recursion} we can proceed in a similar way as in the proof of the closely related formula from \cite[Theorem 5.3]{L}.
\end{proof}

\begin{defn}[Quantum coefficients]
\label{Def:QCoeffs}
For natural numbers $p,r,n$ with $p+r\leq n$ we denote the coefficients in Theorem \ref{Thm:QuantumBinomial} by 
\begin{align*}
c_{p,r,n}=\left[n-r\atop p\right]_q\left[n+1-p\atop r\right]_q.
\end{align*}
Moreover we set $c_{0,n+1,n}=1$ and $c_{p,r,n}=0$ otherwise. Furthermore, for $p+r\leq n$ we put
\begin{align*}
d_{p,r,n}=\left[n-r\atop p\right]_q\left[n-p\atop r\right]_q
\end{align*}
and set $d_{p,r,n}=0$ for $p+r>n$.
\end{defn}

\begin{rem}
The ring homomorphism $\mathbb{Z}[q^{\pm 1/2}]\to \mathbb{Z}[q^{\pm 1/2}]$ defined by $q^{1/2}\mapsto q^{-1/2}$ is called bar involution. It is usually denoted by $\overline{(\cdot)}$. Theorem \ref{Thm:QuantumBinomial} implies that the coefficients in the quantum Laurent expansion with respect to Berenstein--Zelevinsky's basis elements are bar invariant, that is, $\overline{c_{p,r,n}}=c_{p,r,n}$ for all $p,r,n$. For a related result about $F$-polynomials see Tran \cite[Corollary 6.5]{T}. For related results about quantum cluster varieties see Allegretti--Kim \cite[Theorem 1.2 (4)]{AK} and Allegretti \cite[Theorem 4.7 (3)]{A}.
\end{rem}

\subsection{Expansion formulae via perfect matchings}

\begin{notation}
Let us denote the $2n+1$ tiles of the snake graph $\Gn$ in Definition \ref{Defn:KroneckerSnake} by 
\begin{align*}
G_{-n}, G_{-(n-1)},\ldots, G_{n-1}, G_n
\end{align*}
from left to right. Similarly, denote the $2n$ tiles of the snake graph $\SH_n$ by
\begin{align*}
H_{-n}, H_{-(n-1)},\ldots, H_{n-2}, H_{n-1}
\end{align*}
from left to right.
\end{notation}

\begin{defn}[Exponents of tiles] Let $n\geq 0$.
\begin{enumerate}
\item The function $\alpha$ assigns every tile of $\Gn$ the half-integer $\alpha(G_i)=i/2$ if $G_i$ has weight $1$, and the half-integer $\alpha(G_i)=-i/2$ if $G_i$ has weight $2$.
\item The function $\alpha$ assigns every tile of $\SH_n$ the half-integer $\alpha(H_i)=(i-1)/2$ if $H_i$ has weight $1$, and the half-integer $\alpha(H_i)=-i/2$ if $H_i$ has weight $2$. 
\end{enumerate}
\end{defn}

An example is shown in Figure \ref{Figure:Alpha}. For every Kronecker snake graph $\SG$ (e.g. $\SG=\SG_n$ or $\SG=\SH_n$ for some $n\geq 1$) define  a map
\begin{align*}
\alpha\colon\Match(\SG)\to\frac{1}{2}\mathbb{Z},\quad
P\mapsto\sum_{G \in \operatorname{Twist}(P)} \alpha(G).
\end{align*}

\begin{figure}
\begin{center}
\begin{tikzpicture}

\newcommand{\dist}{0.9cm} 
\newcommand{\vdist}{2cm} 


\node (-1down) at (-1*\dist,0) {$\bullet$};
\node (-2down) at (-2*\dist,0) {$\bullet$};
\node (-3down) at (-3*\dist,0) {$\bullet$};
\node (-4down) at (-4*\dist,0) {$\bullet$};
\node (-5down) at (-5*\dist,0) {$\bullet$};
\node (-6down) at (-6*\dist,0) {$\bullet$};
\node (-7down) at (-7*\dist,0) {$\bullet$};
\node (-8down) at (-8*\dist,0) {$\bullet$};

\node (-1up) at (-1*\dist,\dist) {$\bullet$};
\node (-2up) at (-2*\dist,\dist) {$\bullet$};
\node (-3up) at (-3*\dist,\dist) {$\bullet$};
\node (-4up) at (-4*\dist,\dist) {$\bullet$};
\node (-5up) at (-5*\dist,\dist) {$\bullet$};
\node (-6up) at (-6*\dist,\dist) {$\bullet$};
\node (-7up) at (-7*\dist,\dist) {$\bullet$};
\node (-8up) at (-8*\dist,\dist) {$\bullet$};

\path[draw,-,above] (-1*\dist,\dist) edge node {\tiny{$2$}} (-2*\dist,\dist);
\path[draw,-,above] (-2*\dist,\dist) edge node {\tiny{$1$}} (-3*\dist,\dist);
\path[draw,-,above] (-3*\dist,\dist) edge node {\tiny{$2$}} (-4*\dist,\dist);
\path[draw,-,above] (-4*\dist,\dist) edge node {\tiny{$1$}} (-5*\dist,\dist);
\path[draw,-,above] (-5*\dist,\dist) edge node {\tiny{$2$}} (-6*\dist,\dist);
\path[draw,-,above] (-6*\dist,\dist) edge node {\tiny{$1$}} (-7*\dist,\dist);
\path[draw,-,above] (-7*\dist,\dist) edge node {\tiny{$2$}} (-8*\dist,\dist);

\path[draw,-,below] (-1*\dist,0) edge node {\tiny{$2$}} (-2*\dist,0);
\path[draw,-,below] (-2*\dist,0) edge node {\tiny{$1$}} (-3*\dist,0);
\path[draw,-,below] (-3*\dist,0) edge node {\tiny{$2$}} (-4*\dist,0);
\path[draw,-,below] (-4*\dist,0) edge node {\tiny{$1$}} (-5*\dist,0);
\path[draw,-,below] (-5*\dist,0) edge node {\tiny{$2$}} (-6*\dist,0);
\path[draw,-,below] (-6*\dist,0) edge node {\tiny{$1$}} (-7*\dist,0);
\path[draw,-,below] (-7*\dist,0) edge node {\tiny{$2$}} (-8*\dist,0);

\path[draw,-] (-1*\dist,0) edge node {} (-1*\dist,\dist);
\path[draw,-] (-2*\dist,0) edge node {} (-2*\dist,\dist);
\path[draw,-] (-3*\dist,0) edge node {} (-3*\dist,\dist);
\path[draw,-] (-4*\dist,0) edge node {} (-4*\dist,\dist);
\path[draw,-] (-5*\dist,0) edge node {} (-5*\dist,\dist);
\path[draw,-] (-6*\dist,0) edge node {} (-6*\dist,\dist);
\path[draw,-] (-7*\dist,0) edge node {} (-7*\dist,\dist);
\path[draw,-] (-8*\dist,0) edge node {} (-8*\dist,\dist);

\node (-T8) at (-9*\dist,0.5*\dist) {$G$};
\node (-T7) at (-7.5*\dist,0.5*\dist) {$1$};
\node (-T6) at (-6.5*\dist,0.5*\dist) {$2$};
\node (-T5) at (-5.5*\dist,0.5*\dist) {$1$};
\node (-T4) at (-4.5*\dist,0.5*\dist) {$2$};
\node (-T3) at (-3.5*\dist,0.5*\dist) {$1$};
\node (-T2) at (-2.5*\dist,0.5*\dist) {$2$};
\node (-T1) at (-1.5*\dist,0.5*\dist) {$1$};


\node (-T8) at (-9*\dist,-\vdist) {$\alpha(G)$};
\node (-T7) at (-7.5*\dist,-\vdist) {$q^{-3/2}$};
\node (-T6) at (-6.5*\dist,-\vdist) {$q^{1}$};
\node (-T5) at (-5.5*\dist,-\vdist) {$q^{-1/2}$};
\node (-T4) at (-4.5*\dist,-\vdist) {$q^0$};
\node (-T3) at (-3.5*\dist,-\vdist) {$q^{1/2}$};
\node (-T2) at (-2.5*\dist,-\vdist) {$q^{-1}$};
\node (-T1) at (-1.5*\dist,-\vdist) {$q^{3/2}$};


\node (-A7) at (-7.5*\dist,-0.5*\vdist) {$\downmapsto$};
\node (-A6) at (-6.5*\dist,-0.5*\vdist) {$\downmapsto$};
\node (-A5) at (-5.5*\dist,-0.5*\vdist) {$\downmapsto$};
\node (-A4) at (-4.5*\dist,-0.5*\vdist) {$\downmapsto$};
\node (-A3) at (-3.5*\dist,-0.5*\vdist) {$\downmapsto$};
\node (-A2) at (-2.5*\dist,-0.5*\vdist) {$\downmapsto$};
\node (-A1) at (-1.5*\dist,-0.5*\vdist) {$\downmapsto$};

\end{tikzpicture}\hspace{0.8cm}
\begin{tikzpicture}

\newcommand{\dist}{0.9cm} 
\newcommand{\vdist}{2cm} 


\node (-1down) at (-1*\dist,0) {$\bullet$};
\node (-2down) at (-2*\dist,0) {$\bullet$};
\node (-3down) at (-3*\dist,0) {$\bullet$};
\node (-4down) at (-4*\dist,0) {$\bullet$};
\node (-5down) at (-5*\dist,0) {$\bullet$};
\node (-6down) at (-6*\dist,0) {$\bullet$};
\node (-7down) at (-7*\dist,0) {$\bullet$};

\node (-1up) at (-1*\dist,\dist) {$\bullet$};
\node (-2up) at (-2*\dist,\dist) {$\bullet$};
\node (-3up) at (-3*\dist,\dist) {$\bullet$};
\node (-4up) at (-4*\dist,\dist) {$\bullet$};
\node (-5up) at (-5*\dist,\dist) {$\bullet$};
\node (-6up) at (-6*\dist,\dist) {$\bullet$};
\node (-7up) at (-7*\dist,\dist) {$\bullet$};

\path[draw,-,above] (-1*\dist,\dist) edge node {\tiny{$1$}} (-2*\dist,\dist);
\path[draw,-,above] (-2*\dist,\dist) edge node {\tiny{$2$}} (-3*\dist,\dist);
\path[draw,-,above] (-3*\dist,\dist) edge node {\tiny{$1$}} (-4*\dist,\dist);
\path[draw,-,above] (-4*\dist,\dist) edge node {\tiny{$2$}} (-5*\dist,\dist);
\path[draw,-,above] (-5*\dist,\dist) edge node {\tiny{$1$}} (-6*\dist,\dist);
\path[draw,-,above] (-6*\dist,\dist) edge node {\tiny{$2$}} (-7*\dist,\dist);

\path[draw,-,below] (-1*\dist,0) edge node {\tiny{$1$}} (-2*\dist,0);
\path[draw,-,below] (-2*\dist,0) edge node {\tiny{$2$}} (-3*\dist,0);
\path[draw,-,below] (-3*\dist,0) edge node {\tiny{$1$}} (-4*\dist,0);
\path[draw,-,below] (-4*\dist,0) edge node {\tiny{$2$}} (-5*\dist,0);
\path[draw,-,below] (-5*\dist,0) edge node {\tiny{$1$}} (-6*\dist,0);
\path[draw,-,below] (-6*\dist,0) edge node {\tiny{$2$}} (-7*\dist,0);

\path[draw,-] (-1*\dist,0) edge node {} (-1*\dist,\dist);
\path[draw,-] (-2*\dist,0) edge node {} (-2*\dist,\dist);
\path[draw,-] (-3*\dist,0) edge node {} (-3*\dist,\dist);
\path[draw,-] (-4*\dist,0) edge node {} (-4*\dist,\dist);
\path[draw,-] (-5*\dist,0) edge node {} (-5*\dist,\dist);
\path[draw,-] (-6*\dist,0) edge node {} (-6*\dist,\dist);
\path[draw,-] (-7*\dist,0) edge node {} (-7*\dist,\dist);

\node (-T6) at (-6.5*\dist,0.5*\dist) {$1$};
\node (-T5) at (-5.5*\dist,0.5*\dist) {$2$};
\node (-T4) at (-4.5*\dist,0.5*\dist) {$1$};
\node (-T3) at (-3.5*\dist,0.5*\dist) {$2$};
\node (-T2) at (-2.5*\dist,0.5*\dist) {$1$};
\node (-T1) at (-1.5*\dist,0.5*\dist) {$2$};


\node (-T6) at (-6.5*\dist,-\vdist) {$q^{-1}$};
\node (-T5) at (-5.5*\dist,-\vdist) {$q^{1}$};
\node (-T4) at (-4.5*\dist,-\vdist) {$q^0$};
\node (-T3) at (-3.5*\dist,-\vdist) {$q^{0}$};
\node (-T2) at (-2.5*\dist,-\vdist) {$q^{1}$};
\node (-T1) at (-1.5*\dist,-\vdist) {$q^{-1}$};


\node (-A6) at (-6.5*\dist,-0.5*\vdist) {$\downmapsto$};
\node (-A5) at (-5.5*\dist,-0.5*\vdist) {$\downmapsto$};
\node (-A4) at (-4.5*\dist,-0.5*\vdist) {$\downmapsto$};
\node (-A3) at (-3.5*\dist,-0.5*\vdist) {$\downmapsto$};
\node (-A2) at (-2.5*\dist,-0.5*\vdist) {$\downmapsto$};
\node (-A1) at (-1.5*\dist,-0.5*\vdist) {$\downmapsto$};

\end{tikzpicture}
\end{center}

\caption{The map $\alpha$ for $\SG_3$ (left) and $\SH_3$ (right)}
\label{Figure:Alpha}
\end{figure}

\begin{defn}[Expansions from matchings]
\label{Def:Expansion}
For $n\geq 0$ we put
\begin{align*}
r_n=\sum_{P\models \SG_n} q^{\alpha(P)} M[\nu(P)],&&s_n=\sum_{P\models \SH_n} q^{\alpha(P)} M[\nu(P)].
\end{align*}
\end{defn}

Notice that $\alpha(P)=0$ for every perfect matching $P$ of $\SH_1$. Hence the formula for $s_1$ in Definition \ref{Def:Expansion} agrees with the formula for $s_1$ given in Definition \ref{Def:s1}. Note that $s_0=1$ since $\SH_0$ (the snake graph with no tiles consisting of a single vertical unweighted edge) has exactly $1$ perfect matching $P$ with $\nu(P)=0$ and $\alpha(P)=0$.

\begin{lemma}
\label{Lemma:Rec}
The following recursive formulae are true.
\begin{align*}
(a) \quad r_nx_1 = q^{1/2}s_ny_1+r_{n-1}x_2,&&(b)\quad x_2s_n=r_{n-1}+q^{-1/2}M[1,0,1,1]s_{n-1}.
\end{align*}
\end{lemma}

\begin{proof} By definition we have
\begin{align}
\label{Eqn:ProofRecursion}
r_nx_1&=\left(\sum_{P\models \SG_n}q^{\alpha(P)}M[\nu(P)]\right)M[1,0,0,0]=\sum_{P\models \SG_n}q^{\alpha(P)}M[\nu(P)]M[1,0,0,0].
\end{align}
Notice that the minimum perfect matching of $\SG_n$ satisfies $\nu(P_{min})=\left(-n-1,n+2,0,0\right)^T$. 
Here, the first two entries are given by the $g$-vector $(-n-1,n+2)$ of the cluster variable, see equations (\ref{Eqn:gvector}) and (\ref{Eqn:gvectorexplicit}), and the last two entries are zero because $P_{min}$ does not admit twisted tiles with weights $1$ or $2$. Whenever we twist a tile of weight $1$ we loose two edges of weight $2$. Whenever we twist a tile of weight $2$ we gain two edges of weight $1$. Hence every perfect matching satisfies  
\begin{align*}
\nu(P)=\left(-n-1+2y_2(P),n+2-2y_1(P),y_1(P),y_2(P)\right)^T.
\end{align*}
Here we use $y_i(P)$ as a shorthand notation for $\lvert\operatorname{Twist}(P)_i\rvert$ for $i\in\{1,2\}$. From this we can conclude that
\begin{align*}
\frac{1}{2} \nu(P)^T \Lambda \left(\begin{matrix}1\\0\\0\\0\end{matrix}\right) =\frac{1}{2} \nu(P)^T\left(\begin{matrix}
0&0&-1&0\\
0&0&0&-1\\
1&0&0&-2\\
0&1&2&0
\end{matrix}\right)\left(\begin{matrix}1\\0\\0\\0\end{matrix}\right) =\frac{1}{2}y_1(P).
\end{align*}
This term arises as the exponent in the multiplication of the two basis elements in equation (\ref{Eqn:ProofRecursion}). It follows that
\begin{align*}
r_nx_1=\sum_{P\models \SG_n}q^{\alpha(P)+1/2 y_1(P)}M[\nu(P)+(1,0,0,0)^T].
\end{align*}

We split the sum into two parts. Let $e$ be the rightmost vertical edge of $\SG_n$. The map $P\mapsto P'=P\backslash \{e\}$ induces a bijection between perfect matchings of $\SG_n$ that contain $e$ and perfect matchings of the graph $\SH_n$. For such a perfect matching $P$ the rightmost tile of weight $1$ is necessarily twisted. We obtain
\begin{align*}
&\sum_{\substack{P\models \SG_n\\ e\in P}}q^{\alpha(P)+\frac12 y_1(P)}M[\nu(P)+(1,0,0,0)^T]\\
&=\sum_{P'\models \SH_n}q^{\alpha(P'\cup \{e\})+\frac12 \lvert \operatorname{Twist}_1(P'\cup\{e\})\rvert}M[\nu(P)+(1,0,0,0)^T]\\
&=\sum_{P'\models \SH_n}q^{\alpha(P'\cup \{e\})+\frac12 \lvert \operatorname{Twist}_1(P')\rvert+\frac12}M[\nu(P')+(0,0,1,0)^T]\\
&=\left(\sum_{P'\models \SH_n}q^{\alpha(P'\cup \{e\})+\frac12 \lvert \operatorname{Twist}_1(P')\rvert+\frac12 x_1(P')-y_2(P')}M[\nu(P')]\right)M[0,0,1,0]
\end{align*}
when we denote $\nu(P')=(x_1(P'),x_2(P'),y_1(P'),y_2(P'))^T$. We claim that this term is equal to $q^{1/2}s_ny_1$. To prove the claim it is enough to show that the exponent in the term in the last line of the previous equation equals $\alpha(P')+1/2$. The term $ x_1(P')-2y_2(P')$ is invariant under twisting since we gain two edges of weight $1$ whenever we twist a tile of weight $2$. This expression becomes $-n$ for the minimum matching. Hence every perfect matching $P'$ of $\SH_n$ satisfies $ x_1(P')-2y_2(P')=-n$. We have $\alpha(P)-\alpha(P')=-\frac12 y_1(P')+\frac12(n+1)$ because the map $\alpha$ gets shifted by $-1/2$ for every tile with label $1$ when passing from $P$ to $P'$, see Figure \ref{Figure:Alpha}, and $P$ contains the twisted rightmost tile with exponent $(n+1)/2$. This establishes the claim.

Next we investigate the second part of the sum. If $e$ is not contained in a perfect matching $P$ of $\SG_n$, then $P$ contains the rightmost horizontal edges $e_1$ and $e_2$ of $\SG_n$. The map $P\mapsto P''=P\backslash \{e_1,e_2\}$ induces a bijection between perfect matchings of $\SG_n$ that contain $e_1$ and $e_2$ and perfect matchings of the graph $\SG_{n-1}$. For such a perfect matching $P$ the two rightmost tiles with weights $1$ and $2$ are not twisted. We obtain
\begin{align*}
&\sum_{\substack{P\models \SG_n\\ e_1,e_2\in P}}q^{\alpha(P)+\frac12 y_1(P)}M[\nu(P)+(1,0,0,0)^T]\\
&=\sum_{P''\models \SG_{n-1}}q^{\alpha(P''\cup \{e_1,e_2\})+\frac12 \lvert \operatorname{Twist}_1(P''\cup\{e_1,e_2\})\rvert}M[\nu(P)+(1,0,0,0)^T]\\
&=\sum_{P''\models \SG_{n-1}}q^{\alpha(P''\cup \{e_1,e_2\})+\frac12 \lvert \operatorname{Twist}_1(P'')\rvert}M[\nu(P'')+(0,1,0,0)^T]\\
&=\left(\sum_{P''\models \SG_{n-1}}q^{\alpha(P''\cup \{e_1,e_2\})+\frac12 \lvert \operatorname{Twist}_1(P'')\rvert-\frac12 y_2(P'')}M[\nu(P'')]\right)M[0,1,0,0].
\end{align*}
We claim that this term is equal to $r_{n-1}x_2$. We show that the exponent of the term in the last line of the previous equation equals $\alpha(P'')$. By construction $y_2(P)=\lvert \operatorname{Twist}_2(P'')\rvert$. We have $\alpha(P)-\alpha(P'')=-\frac12 \lvert \operatorname{Twist}_1(P'')\vert+\frac12 \lvert \operatorname{Twist}_1(P'')\vert$ because the map $\alpha$ gets shifted by $\pm 1/2$ (depending on the weight) when passing from $P$ to $P''$.

This completes the proof of part (a). Part (b) is shown in a similar way.
\end{proof}

\begin{defn}[Coefficients from perfect matchings]
\label{Def:CoeffsFromMatchings}
Let $p,r,n$ be natural numbers such that $p\leq n+1$ and $r\leq n$. We put
\begin{align*}
\widetilde{c}_{p,r,n}=\sum_{\substack{P\models \SG_n\\y_1(P)=n+1-r\\y_2(P)=p}} q^{\alpha(P)},&&\widetilde{d}_{p,r,n}=\sum_{\substack{P\models \SH_n\\y_1(P)=n-r\\y_2(P)=p}} q^{\alpha(P)}.
\end{align*}
\end{defn}

\begin{rem}
\label{Rem:FormulaRS}
Let $n\geq 0$. A combination of Definitions \ref{Def:Expansion} and \ref{Def:CoeffsFromMatchings}  yields
\begin{align*}
r_n=\sum_{p,r\geq 0} \widetilde{c}_{p,n+1-r,n} M[2p-n-1,2r-n,n+1-r,p],&&s_n=\sum_{p,r\geq 0} \widetilde{d}_{p,n-r,n} M[2p-n,2r-n,n-r,p].
\end{align*}
\end{rem}

\begin{lemma} The following recursions hold for all $p,r,n$.
\label{Lemma:CoefficientRecursions}
\begin{align*}
&(a)\quad q^{\frac{n+1-r}{2}}\widetilde{c}_{p,n+1-r,n}=q^{\frac{n+1}{2}}\widetilde{d}_{p,n-r,n}+q^{\frac{p}{2}}\widetilde{c}_{p,n+1-r,n-1}\\
&(b)\quad q^{-\frac{p}{2}} \widetilde{d}_{p,n-r,n}=\widetilde{c}_{p,n-r,n-1}+q^{-\frac{n-r}{2}}\widetilde{d}_{p-1,n-r-1,n-1}\\
&(c)\quad q^{\frac{n+1-r}{2}} c_{p,r,n}=q^{\frac{n+1}{2}}d_{p,r,n}+q^{\frac{p}{2}}c_{p,n+1-r,n-1}\\
&(d)\quad q^{-\frac{p}{2}}d_{p,r,n}=c_{p,r,n-1}+q^{-\frac{n-r}{2}}d_{p-1,r,n-1}
\end{align*}
\end{lemma}

\begin{proof}
We substitute the expressions from Remark \ref{Rem:FormulaRS} in Lemma \ref{Lemma:Rec} and compare coefficients. Evaluation of Lemma \ref{Lemma:Rec} (a) at $M[2p-n,2r-n,n+1-r,p]$ yields part (a) of this lemma, evaluation of Lemma \ref{Lemma:Rec} (b) at $M[2p-n,2r-n+1,n-r,p]$ yields part (b) of this lemma.

For part (c) we use the quantum Pascal rule to get
\begin{align*}
c_{p,r,n}=\left[n-r\atop p\right]_q\left(q^{\frac{r}{2}}\left[n-p\atop r\right]_q+q^{-\frac{n+1-p-r}{2}}\left[n-p\atop r-1\right]_q\right)=q^{\frac{r}{2}} d_{p,r,n}+q^{-\frac{n+1-p-r}{2}}c_{p,r-1,n-1}.
\end{align*}
Part (d) is proved in the same way as part (c).
\end{proof}

The following theorem is a generalization of Theorem \ref{Thm:SnakeGraphFormula} to the quantum cluster algebra of Kronecker type $\A_q(B,\Lambda)$.

\begin{theorem}
\label{Thm:KroneckerSnakes}
For every $n\geq 0$ the following equality holds true:
\begin{align*}
x_{n+3}&=\sum_{p,r\geq 0} \widetilde{c}_{p,n+1-r,n} M[2p-n-1,2r-n,n+1-r,p]=\sum_{\PM\models\Gn} q^{\alpha(P)} M\left[v(\PM)\right];\\
s_n&=\sum_{p,r\geq 0} \widetilde{d}_{p,n-r,n} M[2p-n,2r-n,n-r,p]=\sum_{P\models \mathcal{H}_n}q^{\alpha(P)}M[\nu(P)].
\end{align*}
Moreover, $s_n\in\A_q(B,\Lambda)$.
\end{theorem}

\begin{proof}
We have to show that the equalities $c_{p,r,n}=\widetilde{c}_{p,n+1-r,n}$ and $d_{p,r,n}=\widetilde{d}_{p,n-r,n}$ are both true for all natural numbers $p,r,n$. We prove these statements by induction on $n$.

For the base case suppose that $n=0$. In this case the statement for the coefficients $d$ and $\widetilde{d}$ is true because by construction $d_{0,0,0}=\widetilde{d}_{0,0,0}=1$ is the only non-zero coefficient. For the statement for the coefficients $c$ and $\widetilde{c}$, note that the single tile $G$ in the snake graph $\SG_1$ satisfies $\alpha(G)=0$. Hence $\alpha(P)=0$ for both perfect matchings $P\models \SG_1$. This implies $\widetilde{c}_{0,0,0}=\widetilde{c}_{0,1,0}=1$ and $\widetilde{c}_{p,r,0}=0$ otherwise. On the other hand, using $[0]_q=1$ and $[1]_q=1$ in Definition \ref{Def:QCoeffs} we see that $c_{0,0,0}=c_{0,1,0}=1$ and $c_{p,r,0}=0$ otherwise.

 The induction step follows from Lemma \ref{Lemma:CoefficientRecursions}. If the statements are true for $n-1$, then parts (b) and (d) of the lemma imply the statement for the coefficients $d$ and $\widetilde{d}$, and using this equality together with parts (a) and (c) of the lemma we establish the equality of coefficients $c$ and $\widetilde{c}$. 
 
Lemma \ref{Lemma:Rec} (a) implies $s_ny_1\in \A_q(B,\Lambda)$. Since the frozen variable $y_1$ is invertible we can conclude that $s_n$ lies in $\A_q(B,\Lambda)$.    
\end{proof}

\begin{rem}
Some authors do not invert the frozen variables in the definition of a (quantum) cluster algebra. Using the recursions from Lemmas~\ref{Lemma:Recursion} and \ref{Lemma:Rec}, one can prove by induction that $s_n$ belongs to the quantum cluster algebra without inverted frozen variables.
\end{rem}

\subsection{BPS states}

There is a connection between the elements
\begin{align*}
s_n=\sum_{P\models \mathcal{H}_n}q^{\alpha(P)}M[\nu(P)]
\end{align*}
associated to self-crossing arcs and Bogomol'nyi--Prasad--Sommerfield states (BPS states) in supersymmetric 4-dimensional quantum field theories. C\'{o}rdova--Neitzke, see \cite{CN}, consider a supersymmetric quantum field theory of quiver type whose line defects $W_n$ are parametrized by natural numbers $n$. The authors consider a generating function $F(W_n)$ of BPS states such that the coefficients are counting BPS states of a given electromagnetic central charge and a given angular momentum (in the $x_3$ direction). 

C\'{o}rdova--Neitzke use a Coulomb branch formula, see \cite[Equation (3.22)]{CN}, to compute $F(W_n)$ recursively. The Coulomb branch formula is based on work of Manschot--Pioline--Sen, see \cite{MPS}, which relates the Coulomb branch to moduli spaces of quiver representations. C\'{o}rdova--Neitzke compute $F(W_n)$ explicitly for $n\in\{0,1,2,3,4\}$. A direct comparison with our expansion formula establishes the following proposition. 

\begin{prop} \label{prop:BPS}
For $n\in\{0,1,2,3,4\}$ we have $s_n=F(W_n)$.
\end{prop}

\begin{question} \label{Conj:BPS}
Does $s_n=F(W_n)$ hold for all $n\in\mathbb{N}$?
\end{question}

Several authors have studied variations of the question in the classical commutative setup. Gaiotto--Moore--Neitzke \cite[Section 5]{GMN} conjectured that if $q=1$, then $F(W_n)$ is equal to a Fock--Goncharov canonical basis element in a cluster variety. C\'{o}rdova--Neitzke give a second method to calculate $F(W_n)$ using the Higgs branch. Allegretti \cite{A} relates generating functions of BPS states calculated via the Higgs branch to elements in cluster algebras.

\section{The Stembridge phenomenon}

Stembridge's $q=-1$ phenomenon \cite[Section 0]{St} asserts the following: Suppose that we have a finite set $B$ of combinatorial origin together with a natural weight function $w\in B\to\mathbb{Z}$. We consider the Laurent polynomial $X(q)=\sum_{b\in B}q^{w(B)}\in\mathbb{Z}[q^{\pm 1}]$ with $X(1)=\lvert B\rvert$. Moreover suppose that $B$ admits a natural involution $\inv\colon B\to B$. Stembridge's phenomenon asserts that often $X(-1)=\lvert B^{\inv}\rvert$ is equal to the number of fixed points in $B$ under the action of $\inv\colon B\to B$.

We consider the quantum cluster algebra attached to the Kronecker quiver from Section~\ref{Section:Kronecker}. Suppose that $p,r,n$ are natural numbers with $p+r\leq n$. The graph $\Gn$ admits a horizontal and a vertical axis of symmetry. We consider the reflection across the vertical axis of symmetry. The reflection induces an involution $\inv\colon \Match(\Gn)\to \Match(\Gn)$. Note that $\inv$ leaves the height monomial $y(\PM)$ of a perfect matching invariant. By restriction we obtain an involution $\inv\colon\Match(\Gn)_{p,r}\to\Match(\Gn)_{p,r}$. To construct the polynomial $X$ let us introduce the following variation of quantum numbers and quantum binomial coefficients.

\begin{defn}[Gaussian integers and binomial coefficients and evaluations at $q=-1$] Let $k,n$ be natural numbers with $k\leq n$.
\begin{itemize}
\item[(a)] The polynomial $(n)_q=(q^n-1)/(q-1)=1+q+\ldots+q^{n-1}\in\mathbb{Z}[q]$ is called \emph{Gaussian integer}. The \emph{Gaussian factorial} is defined as $(n)_q!=(n)_q(n-1)_q\ldots(1)_q\in\mathbb{Z}[q]$. The \emph{Gaussian binomial coefficient} is defined as $({n\atop k})_q=(n)_q!/[(k)_q!(n-k)_q!]\in\mathbb{Z}[q]$.
\item[(b)] The natural numbers $(n)_{-1}$, $(n)_{-1}!$ and $({n\atop k})_{-1}$ are defined to be the evaluations of the polynomials in part (a) at $q=-1$. 
\end{itemize}
\end{defn}

\begin{rem}
The Gaussian integers and quantum integers agree up to a power of $q^{1/2}$. In particular,
\begin{align*}
&\left[n\right]_q=q^{\frac{1-n}{2}} \frac{q^n-1}{q-1}=(n)_q,\quad [n]_q!=q^{-n(n-1)/4}(n)_q!,\quad\left[{n\atop k}\right]_q=q^{-k(n-k)/2} \left({n\atop k}\right)_q,\\
&c_{p,r,n}=q^{-[r+(n-r-p)(p+r)]/2} \left(n-r\atop p\right)_q\left(n+1-p\atop r\right)_q.
\end{align*}
\end{rem}

Let us describe the evaluation of Gaussian binomial coefficients at $q=-1$.

\begin{lemma}
\label{Lemma:EvaluationOfCoeffs}
For all natural numbers $n,k$ with $k\leq n$ the following equation holds:
\begin{align*}
\binom{n}{k}_{-1}
&=\begin{cases}
0&\textrm{if }k\equiv 1\, (\operatorname{mod} 2)\textrm{ and }n\equiv 0\,(\operatorname{mod} 2);\\
\binom{\lfloor n/2\rfloor}{\lfloor k/2\rfloor}&\textrm{otherwise}.
\end{cases}\\
\end{align*}
\end{lemma}
\begin{proof}
In this proof, all congruences are read modulo $2$. The polynomial $(n)_q\in\mathbb{Z}[q]$ is divisible by $q+1$ if and only if $n$ is even. In this case we may write $(n)_q=(q+1)h_n$ for some polynomial $h_n\in\mathbb{Z}[q]$ with $h_n(-1)=n/2$. If $n$ is odd, then $(n)_{-1}=1$.

Thus, when we write the rational function $({n\atop k})_q$ in lowest terms, the irreducible polynomial $q+1$ occurs with multiplicity $\lfloor n/2\rfloor-\lfloor k/2\rfloor-\lfloor(n-k)/2\rfloor$. This multiplicity is equal to $1$ if $k\equiv 1$ and $n\equiv 0$, and it is $0$ otherwise. In particular, $\left(n\atop k\right)_{-1}=0$ if $(k,n)\equiv (1,0)$. For $(k,n)\not\equiv(1,0)$ $(\operatorname{mod} 2)$ we may evaluate $\left(n\atop k\right)_{q}=(n)_q!/[(k)_q!(n-k)_q!]$ at $q=-1$ by replacing every $(l)_q$ with $1$ for every odd $l$ and every $(l)_q$ with $l/2$ for every even $l$. The desired equality follows.
\end{proof}

The following corollary is a direct consequence of Lemma \ref{Lemma:EvaluationOfCoeffs}.

\begin{cor}
\label{Cor:QuantumCoeffsAtMinus1}
Suppose that $p,r,n$ are natural numbers with $p+r\leq n$. If $(p,r,n)$ is congruent to $(1,0,0)$, $(1,1,1)$, $(0,1,1)$, or $(1,1,0)$ modulo $2$, then
\begin{align*}
\left(n-r\atop p\right)_{-1}\left(n+1-p\atop r\right)_{-1}=0.
\end{align*}
If $(p,r,n)$ is congruent to $(0,0,0)$, $(0,1,0)$, $(0,0,1)$, or $(1,0,1)$ modulo $2$, then
\begin{align*}
\left(n-r\atop p\right)_{-1}\left(n+1-p\atop r\right)_{-1}
=\binom{\lfloor (n-r)/2\rfloor}{\lfloor p/2\rfloor}\binom{\lfloor (n+1-p)/2\rfloor}{\lfloor r/2\rfloor}.
\end{align*}
\end{cor}

\begin{theorem}Assume that $p,r,n$ are natural numbers such that $p+r\leq n$. Then 
\begin{align*}
\left(n-r\atop p\right)_{-1}\left(n+1-p\atop r\right)_{-1}=\lvert\left\lbrace\,\PM\in\Match(\Gn)_{p,r}\mid \inv(\PM)=\PM\,\right\rbrace\rvert.
\end{align*}
\end{theorem}

\begin{proof} We distinguish the following cases:
\begin{itemize}
\item[(1)] Assume that $n$ is even. In this case the central tile of $\Gn$ has weight $1$.
\begin{itemize}
\item[(1.1)] Assume that $p$ is even.
\begin{itemize}
\item[(1.1.1)] Assume that $r$ is even. Then for every $\PM\models_p^r\Gn$ the number of tiles $G\in\operatorname{Twist}(G)_1$ of weight $1$, namely $n+1-r$, is odd. Hence the central tile must be twisted.

A perfect matching $\PM\models_p^r\Gn$ with $\inv(\PM)=\PM$ is uniquely determined by a perfect matching of the full subgraph isomorphic to $\SH_{n/2}$ on the left of the central tile (because this determines the structure of $\PM$ on the right of the central tile). Hence the number of matchings $\PM\models_p^r\Gn$ with $\inv(\PM)=\PM$ is equal to the cardinality of the set $\Match(\SH_{n/2})_{p',r'}$ where $p'=p/2$ and $r'=r/2$. Remark \ref{Rem:NumberOf Matchings} and Lemma \ref{Lemma:EvaluationOfCoeffs} imply that $\lvert\left\lbrace\,\PM\in\Match(\Gn)_{p,r}\mid \inv(\PM)=\PM\,\right\rbrace\rvert$ is equal to 
\begin{align*}
\binom{(n-r)/2}{p/2}\binom{(n-p)/2}{r/2}=\left(n-r\atop p\right)_{-1}\left(n+1-p\atop r\right)_{-1}.
\end{align*}
\item[(1.1.2)] Assume that $r$ is odd. Then for every $\PM\models_p^r\Gn$ the number of tiles $G\in\operatorname{Twist}(G)_1$ of weight $1$, namely $n+1-r$, is even. Hence the central tile cannot be twisted, and locally at the central tile the perfect matching contains the top and the bottom edges.

A perfect matching $\PM\models_p^r\Gn$ with $\inv(\PM)=\PM$ is uniquely determined by a perfect matching of the full subgraph isomorphic to $\SG_{(n-2)/2}$ on the left of the central tile (because this determines the structure of $\PM$ on the right of the central tile). Hence the number of matchings $\PM\models_p^r\Gn$ with $\inv(\PM)=\PM$ is equal to the cardinality of the set $\Match(\SG_{(n-2)/2})_{p',r'}$ where $p'=p/2$ and $r'=(r+1)/2$. Remark \ref{Rem:NumberOf Matchings} and Lemma \ref{Lemma:EvaluationOfCoeffs} imply that $\lvert\left\lbrace\,\PM\in\Match(\Gn)_{p,r}\mid \inv(\PM)=\PM\,\right\rbrace\rvert$ is equal to
\begin{align*}
\binom{(n-r-1)/2}{p/2}\binom{(n-p)/2}{(r-1)/2}=\left(n-r\atop p\right)_{-1}\left(n+1-p\atop r\right)_{-1}.
\end{align*}

\end{itemize}
\item[(1.2)] Assume that $p$ is odd. Any perfect matching $\PM\models_p^r\Gn$ with $\inv(\PM)=\PM$ arises from $\PM_{min}$ by twisting the same number of tiles labelled $2$ on each side of the central tile. Hence $\lvert\left\lbrace\,\PM\in\Match(\Gn)_{p,r}\mid \inv(\PM)=\PM\,\right\rbrace\rvert=0$. The desired equality follows from Corollary \ref{Cor:QuantumCoeffsAtMinus1}. 
\end{itemize}
\item[(2)] The case when $n$ is odd can be dealt with in an analogous way. 
\end{itemize}
\end{proof}

\begin{cor}
\label{Cor:Stembridge}
Let $p,r,n$ be natural numbers with $p+r\leq n$. Put
\begin{align*}
X(q)=\left(n-r\atop p\right)_q\left(n+1-p\atop r\right)_q=q^{[r+(n-r-p)(p+r)]/2} \sum_{\PM\models_p^r\Gn} q^{\alpha(P)}
\end{align*}
and define a map $ w\colon \Match(\Gn)_{p,r}\to\mathbb{N}$ by
\begin{align*}
\PM\mapsto \frac{r+(n-r-p)(p+r)}{2}+\alpha(P).
 \end{align*}
Then the quadruple $(\Match(\Gn)_{p,r},\sigma,X,w)$ satisfies Stembridge's $q=-1$ phenomenon.
\end{cor}

\end{document}